\pdfoutput=1
\documentclass[11pt]{article}
\usepackage{ifpdf}

\usepackage{vmargin}
\usepackage{times}

\usepackage[T1]{fontenc}
\usepackage{epsfig}
\usepackage{graphicx}
\usepackage{amssymb}
\usepackage{amsmath}
\usepackage{latexsym} 
\usepackage[utf8]{inputenc}
\usepackage{amsfonts}
\usepackage{amsthm}
\usepackage{paralist}
\usepackage{stmaryrd}

\usepackage[colorlinks=true]{hyperref}

\usepackage{microtype}

\theoremstyle{plain}
\newtheorem{corollary}{Corollary}[section]
\newtheorem{theorem}[corollary]{Theorem}
\newtheorem{lemma}[corollary]{Lemma}
\newtheorem{proposition}[corollary]{Proposition}
\newtheorem*{theorem*}{Theorem}
\newtheorem*{lemma*}{Lemma}
\newtheorem*{definition*}{Definition}
\newtheorem*{corollary*}{Corollary}

\theoremstyle{definition}

\newtheorem{Hypo}{Assumption}

\theoremstyle{remark}

\newtheorem*{remark}{Remark}
\newtheorem*{remarks}{Remarks}

\newtheorem{algorithm}{Algorithm}

\newcommand{\R}{\mathbb{R}}
\newcommand{\Z}{\mathcal{Z}}
\newcommand{\proba}{\mathbb{P}}
\newcommand{\N}{\mathbb{N}}
\newcommand{\E}{\mathbb{E}}

\newcommand{\T}{\mathcal{T}}

\newcommand{\M}{\mathcal{M}}
\renewcommand{\L}{\mathcal{L}}
\newcommand{\D}{\mathcal{D}}

\newcommand{\B}{\mathcal{B}}
\newcommand{\Brown}{\mathfrak{B}}
\newcommand*{\GFF}{\mathfrak{G}}
\newcommand*{\MAX}{\mathfrak{M}}
\newcommand*{\Clown}{\mathfrak{D}}
\newcommand*{\Cont}{\mathfrak{C}}
\newcommand{\Slow}{\mathfrak{S}}

\newcommand{\F}{\mathfrak{F}}

\newcommand{\X}{\mathcal{X}}
\newcommand{\1}{\mathbf{1}}
\renewcommand{\P}{\mathcal{P}}

\newcommand{\p}{{\mathfrak p}}
\newcommand{\e}{\varepsilon}

\newcommand{\ply}{\Rightarrow}

\DeclareMathOperator*{\limit}{\longrightarrow}

\DeclareMathOperator*{\SBB}{SB}

\DeclareMathOperator*{\too}{\lhd}

\DeclareMathOperator*{\diam}{diam}

\newcommand{\Lft}{\curvearrowleft}
\newcommand{\Rgt}{\curvearrowright}
\newcommand{\Fnt}{\triangledown} 

\DeclareMathOperator*{\equal}{=}

\newcommand{\dc}{\mathfrak c}
\newcommand*{\badim}{\underline{\mathfrak d}} 
\newcommand*{\badam}{\overline{\mathfrak d}}

\newcommand{\x}{\alpha}
\newcommand{\y}{\beta}
\newcommand{\z}{\gamma}
\newcommand{\connex}{\mathcal C}
\newcommand{\pT}{\bar{\rho}}
\newcommand{\pL}{\mathring{\rho}} 
\begin{document}
\title{Looptree, Fennec, and Snake of ICRT.}
\author{Arthur Blanc-Renaudie}
\date{\today}
\maketitle
\begin{abstract} We introduce a new theory of plane $\R$-tree, to define plane ICRT (inhomogeneous continuum random tree), and its looptree, fennec (a Gaussian free field on the looptree), and snake. We prove that a.s. the looptree is compact, and that a.s. the fennec and snake are continuous.
 We compute the looptree's fractal dimensions, and the fennec and snake's H\"older exponent. Alongside, we define a Gaussian free field on the ICRT, and prove a condition for its continuity.
  In a companion paper \cite{Dsnake}, we prove that the looptrees, fennecs, and snakes of trees with fixed degree sequence converge toward the looptrees, fennecs and snakes of ICRT.
\end{abstract}
\section{Introduction}
\subsection{Motivations and overview of the results}
We construct the looptree, "fennec" (a gaussian free field on the looptree),  and snake of the ICRT. We also compute the fractal dimensions (Minkowski, Packing, Hausdorff) of the looptree, and the H\"older exponents of the fennec and snake. In a companion paper \cite{Dsnake} we show that these objects are  the scaling limits of the looptrees, fennecs, and snakes of "$\D$-trees" (uniform rooted trees with fixed degree sequence $\D$).

Informally, a looptree, introduced by Curien, Kortchemski \cite{CKloop}, is constructed by replacing each vertex of a tree by a cycle of size proportional to its degree, while keeping the tree structure. Then the fennec (for field+snake) is a Gaussian free field on the looptree. Finally, the snake is the real process obtained by turning around the looptree clockwise, and reading the value of the fennec. Those definitions are made formal in Section \ref{PlanICRT}, using a new theory of plane $\R$-tree. 

Those objects are mainly motivated by scaling limits of maps with fixed face degree sequence. Indeed, the bijections of Bouttier, Di Fransesco, Guitter \cite{BDG} and Janson-Stef\'ansson \cite{JS} yield together a bijection between those maps and $\D$-trees with a discrete fennec. It is now well known that the convergence of the snakes implies the tightness of the maps. However, although Le Gall \cite{LGquad} developed a general approach to prove the universality, this question remains open in general.
In the stable case, this key problem is under active investigation by Curien, Miermont, Riera \cite{CMR}. 
We 
      refer to Marzouk \cite{FixedMAP, FixedMAP2,FixedMAP3} for elaborate discussions on the subject.

Let us already mention, that independently, Marzouk \cite{FixedMAP3} also proves scaling limits of looptrees and snakes of $\D$-trees toward objects constructed from processes with exchangeable increments. We strongly believe that both approaches are useful to study the limits, which thus coincide, from the point of views of processes with exchangeable increments, and stick-breaking constructions. 

In this direction, our theory of plane $\R$-tree builds a bridge between those two points of view. Indeed, we extend many discrete notions to $\R$-trees, to construct several real processes directly from the trees. As a result, those processes can now be studied from stick-breaking constructions. This completes the pioneer ideas of Le Gall (see e.g. \cite{Legall}), which allow to construct and study trees from real process. By analogy with this theory, we construct the height process and Lukasiewicz walk of ICRT, and we will study them in a forthcoming paper \cite{ExcursionStick}.
%
 \subsection{Theory of plane $\R$-tree: an overview}
 A discrete plane  tree is a rooted tree with an ordering of the children of each vertex. Aldous \cite{Aldous3} extended this notion to binary continuum trees, by using signs -/+, to construct the height process. The height process is then constructed by Duquesne \cite{Coding} for general order. 
 
 However an order is not enough to construct several objects, and notably the Lukasiewicz walk. This walk was essential in the work of Le Gall, Le Jan \cite{IntroLevy1,IntroLevy2}, and Le Gall, Duquesne \cite{Duquesne1,Duquesne2} \linebreak[2]to study L\'evy trees. For $\P$-trees and ICRT, Aldous, Miermont, Pitman \cite{ExcICRT,ExcICRT2}, developed a similar theory based on processes with exchangeable increments, but those processes are far less understood. 

 Still their construction of the depth first walk of $\P$-trees led us to two intuitions: First, $\P$-trees are infinite. This is confirmed 
  in \cite{Uniform} since $\P$-trees are the limits of $\D$-trees in the "condensation case" (when the largest degrees have the same order as the total degree, see \cite{Uniform} Theorem 5 (a), 6 (a)). Then, this process can be seen as a Lukasiewicz walk. This interpretation of this "half-discrete-half continuum" tree led us to the theory of plane $\R$-tree below:
 
First we define a notion of angle, which can be seen as numbering the children of each vertex. 
With those angles, we can rewrite the discrete definitions of the looptree and fennec for the ICRT. Then, we define a notion of left and right, which we use to define a contour path in the looptree: morally start at the root, then turn around each cycle clockwise, and stop after a complete circuit.
Finally we construct many processes, by composing this path with some functions on the looptree.
  %
 \subsection{Stick breaking and the chaining method} \label{1.3}
Our approach is based on the stick-breaking construction of the ICRT from \cite{ICRT1}, which is adapted from Aldous, Camarri, Pitman \cite{introICRT1, introICRT2}.  Stick-breaking constructions, first introduced by Aldous \cite{Aldous1}, generate a $\R$-tree and are separated in two steps: 
\begin{compactitem}
\item the line $\R^+$ is first cut into the segments ("sticks" or "branches") $[0, Y_1],\, (Y_1,Y_2], (Y_2,Y_3] \dots$
\item  then for every $i\in \N$ the segment $(Y_i,Y_{i+1}]$ is glued at position $Z_i\leq Y_i$.
\end{compactitem}
 In \cite{ICRT1} we study the compactness and fractal dimensions of ICRT. We now use similar methods, which can be split into simple topological/logical arguments, and many uses of the chaining method.
 
 This method has found many applications in concentration theory (see e.g. Talagrand \cite{Talagrand}, or \cite{Massart} Chapter 13), and to study random metric spaces, and notably stick-breaking constructions (see e.g. Aldous \cite{Aldous1} ; Amini, Devroye, Griffiths, Olver \cite{Amini} ; Curien, Haas \cite{CurienHaas} ; S\'enizergues \cite{Seni}).
 
Let us explain its main principle: The goal is to estimate the max of a function $f$ on a space $S$. To this end, consider a sequence of increasing subspaces\footnote{For more complex algorithms one may want to consider general $(X_i)_{i\in \N}$ and a family of functions $f_i$ on $X_i$.} $(S_i)_{i\in \N}$ of $S$ "approximating" $S$, and for every $i\in \N$, a projection $p_i:S_{i+1}\to S_i$. The main idea is that if $(S_i)_{i\in \N}$ are properly chosen:
\[ \max_{x\in S} f(x)\leq \sum_{i\in \N} \max_{x\in S_{i+1}}(f(x)-f(p_i(x))).\]
As a result, one can decompose a complex estimate into many simpler ones. Moreover, when $(S_i)_{i\in \N}$ are well chosen, it tends to give optimal bounds \footnote{Reverse bounds tends to be much harder to prove. See e.g. \cite{ICRT1} for ICRT in the non compact case.}.

In most of our proofs $S$ is the looptree, and $(S_i)_{i\in \N}$ are the sub-looptrees obtained after gluing a certain number of branches. Then $f$ can be the distance between a vertex and a fixed set to prove compactness, or compute fractal dimensions. $f$ can also be some partial sums to prove uniform convergence, or continuity. In \cite{Uniform}, it was also used to prove the tightness of $\D$-trees. 
  
In this paper we often re-decompose for every $i\in \N$, the estimate of $\max_{x\in S_{i+1}} f(x)-f(p_i(x))$: On the one hand, we estimate the maximum number of branches separating $x\in S_{i+1}$ from $S_i$. On the other hand, we estimate how $f$ vary on each branches. Finally we multiply the worst cases. 

 
\paragraph{Plan of the paper:} In Section \ref{PlanICRT}, we define our objects and state our main results. We also deduce the H\"older continuity of the fennec and snake from our other results. In Section \ref{RecallSec}, we recall several technical results from \cite{ICRT1}. In Section \ref{Compacito}, we prove the compactness of the looptree. In Section \ref{LeftSec}, we study the notions of left and right. 
In Section \ref{ContourPathSec}, we construct the contour path and prove its H\"older continuity. In Section \ref{DimensionsSec},  we compute the fractal dimensions of the looptree. In Section \ref{CompactGFFSec}, we study the Gaussian free field on the ICRT. In Section \ref{CompactFennecSec}, we prove that the fennec is well defined, and extend its definition to other fields. Both Sections \ref{CompactGFFSec}, \ref{CompactFennecSec} can be read right after Section \ref{RecallSec}. 

\paragraph{Acknowledgment} I am grateful to Cyril Marzouk for the interesting discussions we got at CIRM. 

%
\section{Model and main results.} \label{PlanICRT}
\subsection{Basic notions on $\R$-trees and plane $\R$-trees.} \label{RtreeDef}
A {\it Polish space} is a separable, complete, metric space.  A {\it $\R$-tree} is a geodesic, loopless, Polish space (see Le Gall \cite{Legall}). A {\it rooted $\R$-tree} is a $\R$-tree with a distinguished vertex.  

For every $\R$-tree $\T$, and $x,y\in \T$,  let $\llbracket x,y \rrbracket$ denote the geodesic path between $x$ and $y$. 
The {\it closest common ancestor} of $x,y\in \T$ is the vertex  $x\wedge y\in \llbracket \rho,x\rrbracket \cap \llbracket \rho,y\rrbracket  $ which maximizes $d(\rho,z)$. For every $x\in \T$, the {\it degree} $\deg(x)$ of $x$ in $\T$ is the number of connected components of $\T\backslash \{x\}$. 

An {\it angle function} on a rooted $\R$-tree $(\T,d,\rho)$ is a function $u: \T^2 \to [0,1]$ such that:
\begin{compactitem}
\item For all $x \in \T$, $u_{x,\rho}=u_{x,x}=0$. 
\item For all $x\in \T$, $y,z \in \T\backslash\{x\}$, $u_{x,y}=u_{x,z}$ iff $y$ and $z$ are connected in $\T\backslash \{x\}$.
\end{compactitem}
{\it A plane $\R$-tree} is a rooted $\R$-tree equipped with an angle function.

To avoid measurability issues, we further assume in the definition that a plane $\R$-tree is {\it balanced}: for every $x,y\in \T$ if  $\deg(x)=2$ then $u_{x,y}\in \{0,1/2\}$.
\subsection{ICRT, and plane ICRT} \label{ICRTDef}
We first introduce a generic stick breaking construction.  It takes for input two sequences in $\R^+$ called cuts ${\mathbf y}=(y_i)_{i\in \N}$ and glue points ${\mathbf z}=(z_i)_{i\in \N}$, which satisfy
\begin{equation*} \forall i<j,\ \ y_i<y_j \qquad ; \qquad y_i\limit \infty \qquad ; \qquad \forall i\in \N,\ \ z_i\leq y_i, \label{2609} \end{equation*}
and creates an $\R$-tree  by recursively "gluing" segment $(y_i,y_{i+1}]$ on position $z_i$ 
,  or rigorously, by constructing  a consistent sequence of distances $(d_n)_{n\in \N}$ on $([0,y_n])_{n\in \N}$.
\begin{algorithm} \label{Alg1} \emph{Generic stick-breaking construction of $\R$-tree.}
\begin{compactitem}
\item[--] Let $d_0$ be the trivial metric  on $[0,0]$.
\item[--] For each $i\geq 0$ define the metric $d_{i+1}$ on $[0, y_{i+1}]$ such that for each $x\leq y$: 
\[ d_{i+1}(x,y):=
\begin{cases} 
d_{i}(x,y) & \text{if } x,y\in [0, y_i] \\
d_{i}(x,z_i)+|y-y_i| & \text{if } x \in [0, y_i], \, y \in (y_i, y_{i+1}] \\
|x-y| &  \text{if } x,y\in (y_i, y_{i+1}]
\end{cases} \]
where by convention $y_0:=0$ and $z_0:=0$.
\item[--] Let $d$ be the unique metric on $\R^+$ which agrees with $d_i$ on $[0, y_i]$ for each $i\in \N$.
\item[--] Let $\SBB({\mathbf y},{\mathbf z})$ be the completion of $(\R^+,d)$.
\end{compactitem}
\end{algorithm}

Let $\Omega$ be the space of sequences $\{\theta_i\}_{i\in \N}$ in $\R^+$ with $\sum_{i=0}^\infty \theta_i^2=1$ and $\theta_1\geq \theta_2 \geq \dots$ The ICRT of parameter $\Theta\in \Omega$ is the random $\R$-tree constructed via the following algorithm. \pagebreak[2]
\begin{algorithm} \label{ICRT} \emph{Construction of $\Theta$-ICRT (from \cite{ICRT1})}
\begin{compactitem}
\item[-] Let $\mathbf X=(X_i)_{i\in \N}$ a family of independent exponential random variables of parameter $(\theta_i)_{i\in \N}$.
\item[-] Let $\mu$ be the measure on $\R^+$ defined by $\mu=\theta_0^2 dx+\sum_{i=1}^{\infty} \delta_{X_i} \theta_i$. 
\item[-] For each $l\in\R^+$ let $\mu_l$ be the restriction of $\mu$ to $[0,l]$, and let $\p_l:=\mu_l/\mu[0,l]$.
\item[-] Let $\mathbf Y= (Y_i)_{i\in \N}$ be a Poisson point process  on $\R^+$ of rate $\mu[0,l]dl$. \item[-] Let $\mathbf Z=(Z_i)_{i\in \N}$ be a family of independent random variables with laws $(\p_{Y_i})_{i\in \N}$.
\item[-] The $\Theta$-ICRT is defined as $(\T,d_\T)=\SBB(\mathbf Y,\mathbf Z)$ (see Algorithm \ref{Alg1}).
\end{compactitem}
\end{algorithm}
\begin{remarks} $\bullet$ When $\theta_0=1$, the ICRT is the Brownian CRT. 
\\ $\bullet$ When $\theta_0=0$ and $\sum_{i=1}^\infty \theta_i<\infty$, $\T$ "is" a $\P$-tree with a modified distance (see \cite{Uniform} Section 5.2). 
\\ $\bullet$ When $\Theta$ is random and corresponds to the jumps and brownian part of a L\'evy bridge, the ICRT "is" a L\'evy tree. (Equal in GP distribution, by unicity of the limit of $\D$-trees, see \cite{Uniform} Section 8.1.)
\\ $\bullet$ Morally, $(X_i)_{i\in \N}$ and  $(\theta_i)_{i\in \N}$ corresponds to the vertices of highest degrees with their degrees. On the other hand, $\theta_0$ corresponds to vertices with small degrees. 
\end{remarks}
We root ICRT at $0$. Recall that to define a plane $\R$-tree, we need an angle function $U$. This is equivalent to define for each $x\in \T$, and for each connected component $C$ of $\T\backslash \{x\}$ with $0\notin C$, the value of $U_{x,y}\in[0,1]$ for a unique $y\in C$. So, the following algorithm a.s. does well define an angle function on the ICRT. Also, the last line below insure that a.s. $(\T,d_\T,0,U)$ is balanced.

\begin{algorithm} \label{ConstructU} Construction of the uniform angle function $U$ on the ICRT.
\begin{compactitem} 
\item[-] Let $(U_{X,i})_{i\in \N}, (U_{Z,i})_{i\in \N}$ be independent uniform random variables in $[0,1]$.
\item[-] Let $U$ be the unique angle function on $\T$ such that:
\begin{compactitem} 
\item[-] For every $i\in \N$, $U(X_i,Y_{\inf\{a\in \N:Y_a>X_i\}})=U_{X,i}$
\item[-] For every $i\in \N$, $U(Z_i,Y_{i+1})=U_{Z,i}$.
\item[-] For every $x\in \R^+\backslash \bigcup_{i\in \N} \{X_i,Y_i\}$, we have $U(x,Y_{\inf\{a\in \N:Y_a>x\}})=1/2$.
\end{compactitem}
\end{compactitem}
\end{algorithm}
\subsection{The ICRT looptree} \label{LoopDef}
To extend the discrete setting, we want to replace each vertex by a loop. So we define $\mathcal L$ as $\T\times [0,1]$ with a proper pseudo-distance $d_\L$ corresponding to the cycles. 
 
To define $d_\L$ we need the sizes of the cycles, which in the discrete correspond to the degrees. For ICRT, although the degrees are infinite, the only vertices with high degrees are $(X_i)_{i\in\N}$ and their degrees are morally proportional to $(\theta_i)_{i\in \N}$. 
 
 Thus we may define $\L$ by concatenating some cycles of perimeter $(\theta_i)_{i\in \N}$. Actually, this is not enough, since we forget the vertices of small degrees. Morally their degrees corresponds to $\theta_0dl$. Then by concatenating the corresponding cycles we get segments of length $\theta^2_0/4 dl$. (The factor $1/4$ is the mean distance between two points in a cycle of perimeter $1$.) 

So, we formally define the ICRT looptree as follows: Let $\dc$ be the distance in the torus $[0,1]$. Then for every  $x,y\in \T,u\in [0,1]$, let $U_{x,y,u}=U_{x,y}$ if $x\neq y$ and let $U_{x,y,u}=u$ otherwise. We define a pseudo-distance $d_\L$ on $\T\times [0,1]$ such that for all $(x,u),(y,v)\in \T\times [0,1]$ (see Figure \ref{looptreefig3}),
\begin{equation} d_{\L}((x,u),(y,v)):=\frac{\theta^2_0}{4}d_\T(x,y)+ \sum_{i\in \N}\theta_i \dc(U_{X_i,x,u},U_{X_i,y,v}). \label{18/12/15h} \end{equation}
Finally let $(\L,d_\L)$ be the completion of the pseudo-metric space $(\T\times [0,1],d_\L)$.

 \begin{figure}[!h] \label{looptreefig3}
\centering
\includegraphics[scale=0.45]{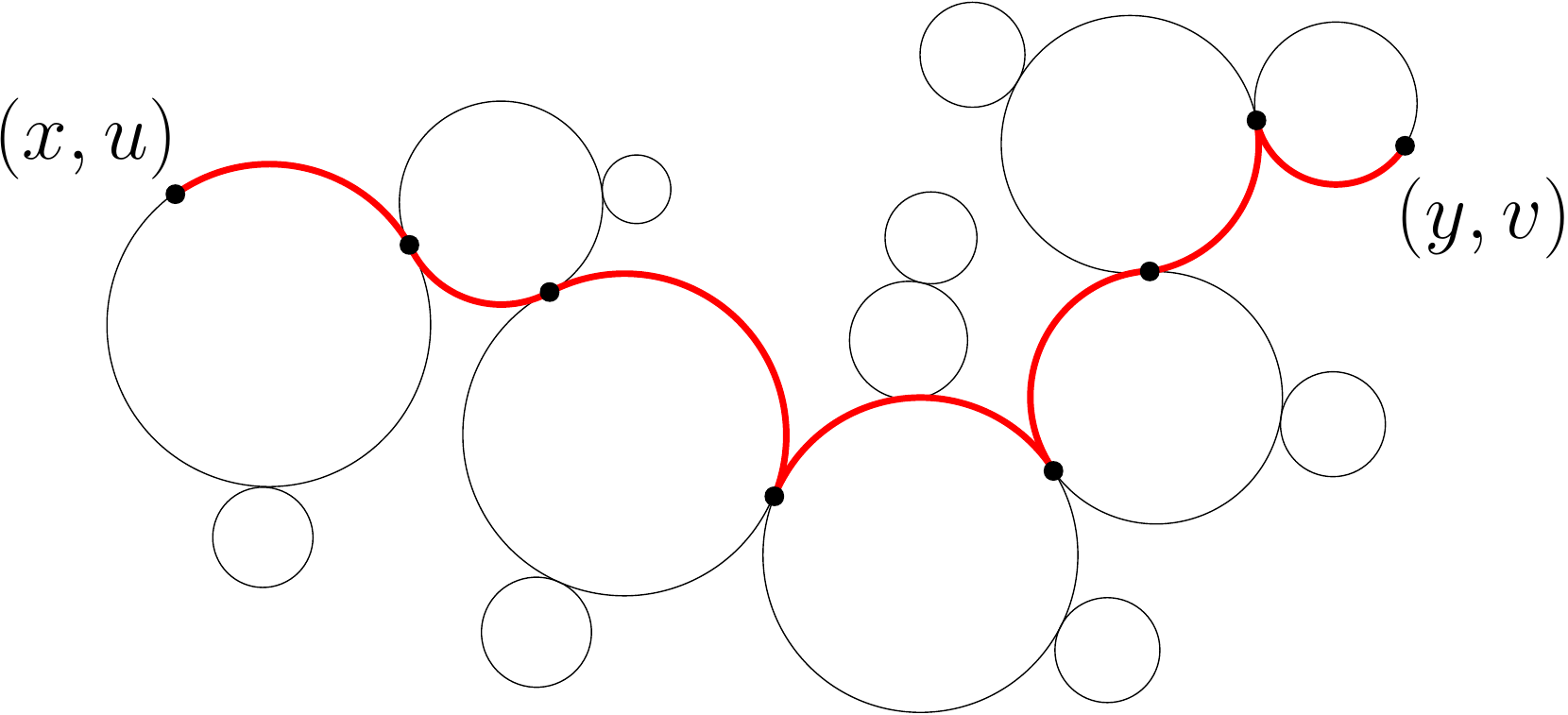}
\caption{A continuum looptree $\L$. The geodesic between $(x,u),(y,v)\in \L$ is red. The distance between $(x,u)$ and $(y,v)$ is the sum of the length of all potential cords that lies in this geodesic. This is a simplified picture, since $\L$ usually has infinitely many cycles which can have size null.
} 
\label{explore1} \label{looptreefig3}
\end{figure}%
\begin{theorem} \label{CompactLooptree} Almost surely $d_\L$ is finite on $\T\times [0,1]$, and $(\L,d_\L)$ is compact.
\end{theorem}
\begin{remark} When $\theta_1=1$, $\L$ is a cycle of size $1$. When $\theta_0=1$, $\L$ is the Brownian CRT.
\end{remark}
Sadly, several of our notions are initially defined in $\R^+\times[0,1]$ or in $\T\times[0,1]$, but not in $\L$. So, to avoid any issues, we extend the above notions to $\L$:
\begin{proposition} \label{Completion} Almost surely the following assertions hold:
\begin{compactitem}
 \item[(a)] $\L$ is the completion of $(\R^+\times[0,1],d_\L)$.
\item[(b)] If $\T$ is compact, then $\L=\T\times[0,1]$, and $p_{\T,\L}:(x,u)\in \L\to x\in \T$ is continuous.
\item[(c)] For every $i\in \N$ with $\theta_i>0$, $\x\in \T\times[0,1] \mapsto  U_{X_i,\x}$ extends to a continuous function from $(\L,d_\L)$ to $([0,1],\dc)$. 
\item[(d)] For every $\x,\y\in \L$, writing when $\theta_0>0$ $d_\T:(x,u),(y,v)\in \L\mapsto d_\T(x,y)$,
\[ d_{\L}(\x,\y)=\frac{\theta^2_0}{4}d_\T(\x,\y)+ \sum_{i\in \N}\theta_i \dc(U_{X_i,\x},U_{X_i,\y}). \]
\end{compactitem}
\end{proposition}
\begin{remark} If $\T$ is not compact, then $\L$ morally contains  the "ends of the infinite branches" of $\T$ (see Section \ref{ExtendLeftSec}). Those points are dense in $\L$ (cf \cite{ICRT1} Lemma 6.6), so $\p_{\T,\L}$ is nowhere continuous.
\end{remark} 

We now give the fractal dimensions of $\L$, whose definitions are recalled in Section \ref{RecallDefDimSec}. Let,
\begin{equation} \badim:=1+\liminf_{l\to \infty}\frac{\log \E[\mu[0,l]]}{\log l} \quad ; \quad \badam:=1+\limsup_{l\to \infty}\frac{\log \E[\mu[0,l]]}{\log l}.\label{22/02/9h}\end{equation} 
\begin{remark} $l\mapsto l\E[\mu[0,l]$ is an analog of the Laplace exponent $\psi$ for L\'evy processes (see \cite{ExcICRT2,ICRT1}). Also, by Lemma \ref{Recall icrt} (b), $1\leq \badim\leq \badam\leq2$.
\end{remark}
\begin{theorem} \label{DimTHM} Almost surely the upper Minkowski dimension, and Packing dimension of $\L$ are $\badam$. Almost surely the lower Minkowski dimension, and Hausdorff dimension of $\L$ are $\badim$
\end{theorem}
%
\subsection{The ICRT fennec} \label{The ICRT fennec}
To mimic the discrete setting, we want to define the ICRT fennec $\F$ as a Gaussian free field (a random function) on $\L$. We construct it explicitly by mimicking our construction of the looptree. 

First, to deal with vertices of small degree we construct a Gaussian free field on $\T$. To this end, we adapt Algorithm \ref{Alg1}: Let $\Brown:\R^+\to \R$ be a Brownian motion. Then define inductively $\GFF$ on $\R^+$ such that for every $i\in \N$ and $Y_i<x\leq Y_{i+1}$, 
\begin{equation}  \GFF(x):=\GFF(Z_i)+\Brown(x)-\Brown(Y_i). \label{22/12/00} \end{equation}
To construct the fennec $\F$, we need to show that $\GFF$ extends to a continuous function on $\T$ if $\theta_0>0$. We actually prove the much stronger result:
\begin{theorem} \label{CompactGFF} Almost surely $\GFF$ extends to a continuous function on $\T$ if
\begin{equation} \int^\infty \frac{dl}{l\sqrt{\E[\mu[0,l]]}}<\infty. \label{22/12/9h} \end{equation}
\end{theorem} 
\begin{remark} By \cite{ICRT1}, $\T$ is almost surely compact iff $\int^\infty dl/(l\E[\mu[0,l]])<\infty$,
 so $\GFF$ is well defined on most compact ICRT. Also, we believe one can adapt \cite{ICRT1} Section 6.3 to show that \eqref{22/12/9h} is necessary. Moreover, in \cite{Duquesne2} Chapter 4.5, Duquesne and Le Gall  prove   the equivalence for L\'evy trees.
 %
\end{remark} 

We now adapt \eqref{18/12/15h}. Beforehand, recall that if $\theta_0>0$, by Proposition \ref{Completion} (b),  that $\L=\T\times[0,1]$ and that $p_{\L,\T}:(x,u)\in \L\to x$ is continuous. 
 Then the fennec is
\begin{equation} \F:\x\in \L\mapsto \frac{\theta_0}{\sqrt{6}}\GFF\circ p_{\L,\T}(\x)+\sum_{i=1}^\infty \sqrt{\theta_i} \Brown_i(U_{X_i,\x}). \label{27/12/18hc} \end{equation}
\begin{theorem} \label{Prop3} \label{CompactFennec} Almost surely the sum in \eqref{27/12/18hc} converges uniformly on $\L$, so $\F$ is continuous.
\end{theorem}
We will actually prove a stronger statement where the functions $(\Brown_i)_{i\in \N}$ are replaced by more general random functions under a moment condition for their maximums (see Section \ref{CompactFennecSec} for details). We believe that this extension may have applications to study more complex fields on $\D$-trees.

A direct corollary of Theorem \ref{Prop3} is that $\F$ is indeed a Gaussian free field on $\L$: 
\begin{proposition} \label{GFF_GOOD} Almost surely, conditionally on $\mathbf X,\mathbf Y,\mathbf Z, U$, for every $\x,\y\in \L$, $\F(\x)-\F(\y)$ is Gaussian with variance 
\[ d'_\L(\x,\y):=\frac{\theta_0^2}{6}d_\T(\x,\y)+\sum_{i\in \N}\theta_i |U_{X_i,\x}-U_{X_i,\y}|(1-|U_{X_i,\x}-U_{X_i,\y}|).\]
\end{proposition}
\begin{remark} $\tfrac{1}{2}d_\L\leq d'_\L\leq d_\L$ so $d'_\L$ is equivalent to $d_\L$.
\end{remark}
\subsection{Left, Front, Right (see Figure \ref{LR2})} \label{DefLeftFrontRight}
In this section $(\T,d,\rho,u)$ denote an arbitrary plane $\R$-tree. For every $x\in \T, (y,w)\in \T\times[0,1]$ let $u_{x,y,w}:=u_{x,y}$ if $x\neq y$ and let $u_{x,y,w}:=w$ otherwise.
For all $\x=(x,v),\y=(y,w)\in \T\times[0,1]$, we say that $\x$ is at the left of $\y$ (or $\y$ is at the right of $\x$) and write $\x\Rgt \y$ if $u_{x\wedge y,x,v}<u_{x\wedge y,y,w}$. 
We say that $\y$ is in front of $\x$ and write $\x\too \y$ if $x\in \llbracket 0, y \rrbracket $ and $u_{x,y,w}=v$.
%
\begin{figure}[!h] 
\centering
\includegraphics[scale=0.6]{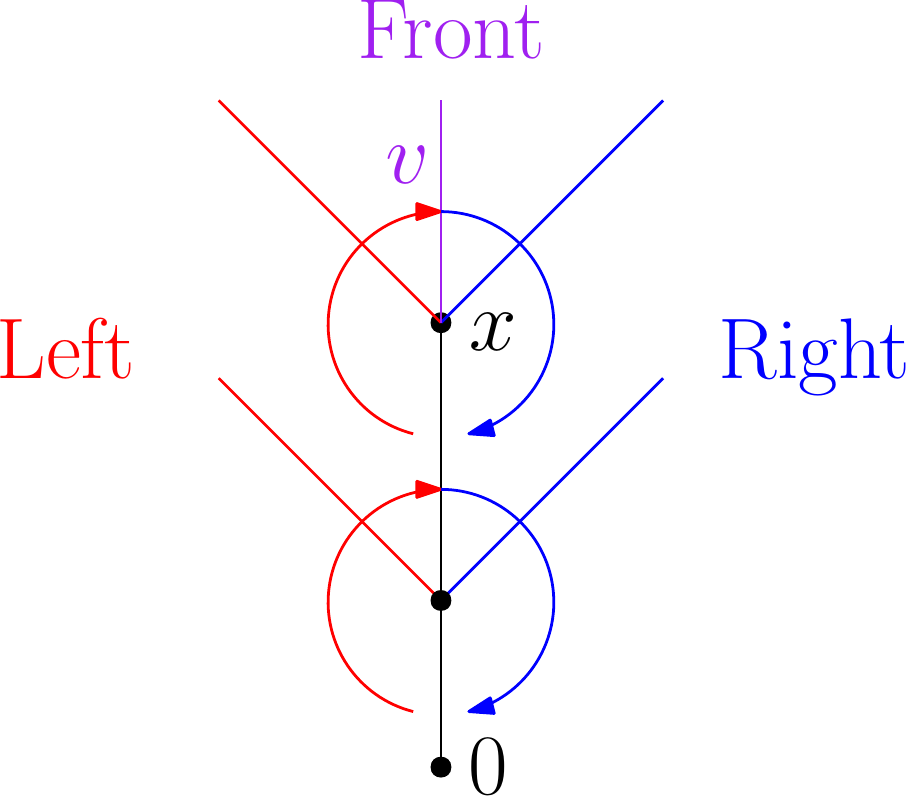}
\caption{A vertex $(x,v)\in\T\times[0,1]$ with its left (red), front (purple), right (blue) are represented. 
}
 \label{LR2}
\end{figure}

\begin{lemma} \label{dfso} \label{deso pas deso} \label{Musique qui tue}  Let $\prec$ be the binary relation defined on $\T\times[0,1]$ such that for every $\x,\y\in \T\times[0,1]$:
 \[ \x\prec \y \iff (\x\Rgt \y) \text{ or } (\text{$\x\too \y$}).  \] 
 Then $\prec$ is a total order relation on $\T\times[0,1]$ and is called the contour order.
\end{lemma}
\begin{proof} See Appendix \ref{SadoMasoSec}. 
\end{proof}
Finally let $\nu$ be a $\sigma$-finite borel measure on $\T$.
Let $\nu_\L:=\mu\times \1_{l\in [0,1]} dl$. The mass on the left, front, right of $\x\in \T\times[0,1]$ are denoted by, (see Appendix \ref{SadoMasoSec} for definiteness)
\[ \nu_{\Lft}(\x):=\nu_\L \{\y: \y\Rgt \x\} \quad ; \quad  \nu_{\Fnt}(\x) : =  \nu_\L\{\y: \x\too \y\} \quad ; \quad \nu_{\Rgt}(\x):= \nu_\L \{\y: \x\Rgt \y\}. \]
\subsection{The contour path, and the ICRT snake}
Recall that for every $l\in \R^+$, $\mu_l$ is the restriction of $\mu$ to $[0,l]$, and that $\p_l=\mu_l/\mu[0,l]$. Also by \cite{ICRT1} a.s. $(\p_l)_{l\in \R^+}$ converges weakly toward a probability measure $\p$. We prove that a.s. $\p_\Lft$ extends to a function continuous at $\L\backslash(\R^+\times[0,1])$ (see Section \ref{ExtendProjSec}), and that $\p_\Lft$ has an "inverse": Below $\sim_{\L}$ denote the equivalent relation on $\L$ such that for every $\x,\y\in \L$, $\x\sim_{\L} \y$ iff $d_\L(\x,\y)=0$.
\begin{theorem} \label{ConstructContThm} Almost surely there exists a continuous function $\Cont:[0,1]\mapsto \L$ such that for every $\x\in \L$, $\Cont(\p_{\Lft}(\x))\sim_\L \x$. We call $\Cont$ the contour path on $\L$ (see Figure \ref{looptreefig5}).
\end{theorem}
\begin{figure}[!h] \label{looptreefig5}
\centering
\includegraphics[scale=0.55]{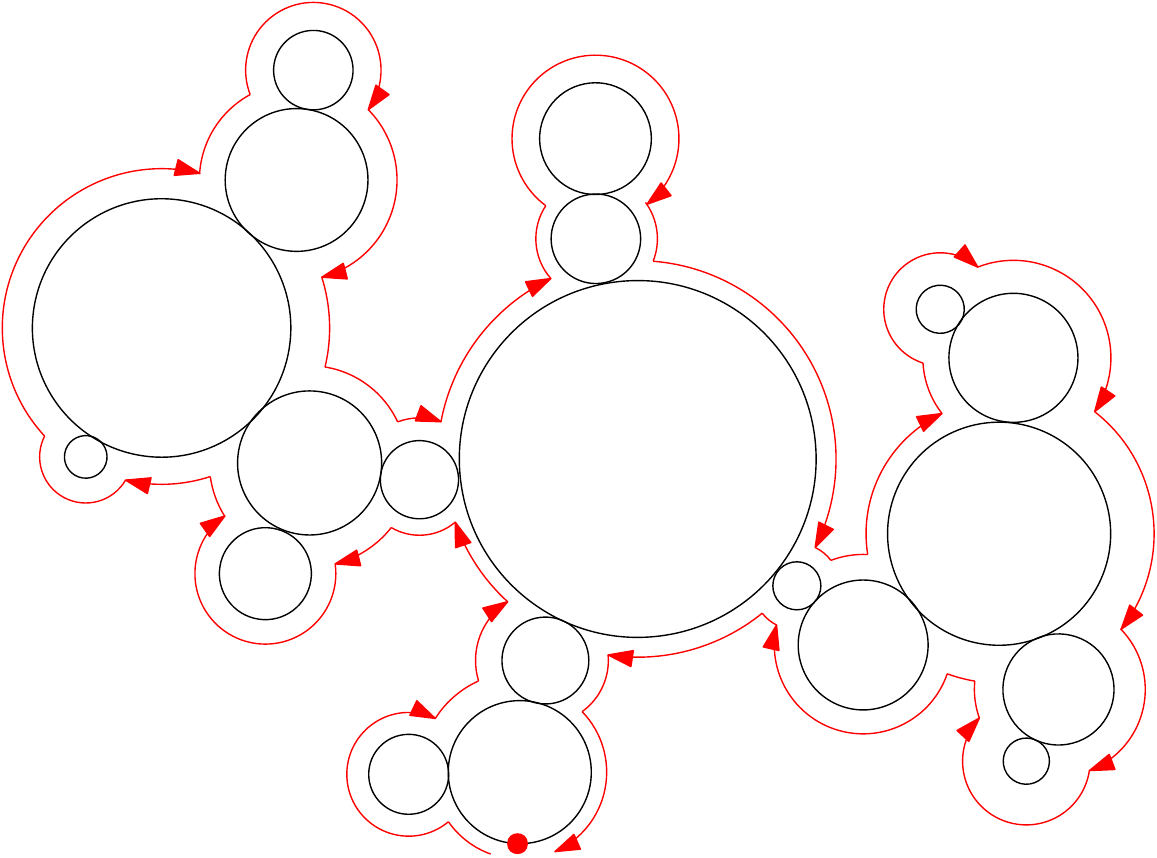}
\caption{A continuum looptree $\L$, with its contour path $\Cont:[0,1]\mapsto \L$ in red. The path start  at the root $(0,0)$ then "turn around" each cycle clockwise. It is continuous, "surjective", not "injective". 
}  
\label{explore1} \label{looptreefig5}
\end{figure}%

The density of $\{ p_\Lft(\alpha),\alpha\in \L\}$ (see Lemma \ref{prelimLeft}) implies that $\Cont$ is unique up to $\sim_\L$. Moreover it also implies with the existence of $\Cont$ that:
\begin{proposition} \label{DefSnake} Almost surely for every continuous function $F:\L\mapsto\R$, $F\circ \Cont$ is the unique continuous function $f$ such that for every $\alpha\in \L$, $f(\p_\Lft(\alpha))=F(\alpha)$.
\end{proposition}
In particular, we define the ICRT snake as $\Z:=\F\circ \Cont $. 
\begin{remark}
Similarly, if $\T$ is a.s. compact, we can define the height process of the ICRT as follows. First by Proposition \ref{Completion} (b), $\L=\T\times[0,1]$, and $p_{\T,\L}:(x,u)\in \L\to x$ is continuous. Then we define the contour path on $\T$ as $\Cont_\T:=p_{\T,\L}\circ \Cont$. This path is a.s. continuous by continuity of $\Cont$. Finally, the height process is
\[ \mathfrak H:x\in [0,1]\mapsto d_\T(0,\Cont_\T(x)). \]

Also by first defining
\[ \mathfrak L:(x,u)\in \L \mapsto \frac{\theta_0^2}{2}d_\T(0,x)+ \sum_{i\in \N:X_i\in \llbracket 0, x\rrbracket}\theta_i (1-U_{X_i,x,u}),\] we may define the Lukasiewicz walk as $\mathfrak X:= \mathfrak L\circ \Cont$. 
We will study $\mathfrak H$ and $\mathfrak X$ in \cite{ExcursionStick}. 
\end{remark}
 

We now consider the H\"older continuity. Our first point is that for any continuous function $F$, the H\"older continuity of $F\circ \Cont$ can be deduced from the H\"older continuity of $F$ and $\Cont$. Moreover:
\begin{theorem} \label{HolderContThm} Recall \eqref{22/02/9h}.  A.s. $\Cont$ is H\"older continuous with any exponent smaller than $1/\badam$.
\end{theorem}
\begin{remark} The bound is optimal. Indeed, if $\Cont$ is $\alpha$-H\"older continuous then $\L$ have Minkowski upper dimension at most $1/\alpha$ (see Lemma \ref{Minkowski<Cont}).
\end{remark}

Thus, since by Proposition \ref{GFF_GOOD}, a.s. $\F$ is a Gaussian free field on $(\L,d'_\L)$, which have by Theorem \ref{DimTHM} finite upper Minkowski dimension, we deduce (see Lemma \ref{HolderXD}):
\begin{theorem}\label{Holder2} Almost surely $\F$ is H\"older continuous with any exponent smaller than $1/2$, and $\Z$ is H\"older continuous with any exponent smaller than $1/2\badam$. 
\end{theorem}

\section{Recalls from \cite{ICRT1} Section 4} \label{RecallSec}
 We prove the following results in \cite{ICRT1} Section 4: (Actually in \cite{ICRT1} we focus on the case where $\theta_0>0$ or $\sum_{i=1}^\infty \theta_i=\infty$, but the proof is exactly the same in the complementary case. Also (d) is slightly modified from Lemma 4.6 using Lemma 4.5.)
\begin{lemma} \label{Recall icrt} The following assertions hold a.s.:
\begin{compactitem}
\item[a)] The map $l\mapsto \E[\mu[0,l]]$ is concave.
\item[b)]As $l\to \infty$, $\mu[0,l]\sim \E[\mu[0,l]]=\theta_0^2l+o(l)$.
\item[c)] For every $l$ large enough, $\#([0,l]\cap\{Y_i\}_{i\in \N})\leq 2\mu[0,l]l\leq 2l^2$.
\item[d)] For every $i$ large enough, $\mu[Y_{i-1},Y_i]\leq \log(Y_i)^2/Y_i$.
\end{compactitem}
\end{lemma}
We now adapt \cite{ICRT1} Lemma 4.7. First let us define a metric, which morally counts the number of branches between two points. To do so, we adapt Algorithm \ref{Alg1}: First let $d_{\mathfrak N,0}$ be the trivial metric on $[0,0]$. Then for every $i\geq 0$, define $d_{\mathfrak N,i+1}$ as the metric on $[0,Y_{i+1}]$ such that for every $x< y$,
\[ d_{\mathfrak N,i+1}(x,y):=
\begin{cases} 
d_{\mathfrak N,i}(x,y) & \text{if } x,y\in [0, Y_i] \\
d_{\mathfrak N,i}(x,Z_i)+1& \text{if } x \in [0, y_i], \, y \in (Y_i, Y_{i+1}] \\
1 &  \text{if } x,y\in (Y_i, Y_{i+1}]
\end{cases} \]
Finally let $d_{\mathfrak N}$ be the unique metric on $\R^+$ which agrees with $d_{\mathfrak N,i}$ on $[0, Y_i]$ for every $i\in \N$.

\begin{lemma}\label{CountingLemma} Almost surely, for every $n$ large enough, for every $x\leq 2^{n+1}$, $d_{\mathfrak N} (x,[0,2^n])\leq 4n$.
\end{lemma}
\begin{proof} We adapt the proof of \cite{ICRT1} Lemma 4.7. Let $\mathcal F:=\sigma(\mu,\{Y_i\}_{i\in \N})$. Let $z$ be a random variable $\mathcal F$-measurable in $[0,2^{n+1}]$. Let us follow the geodesic path between $z$ and $[0,2^n]$ (see Figure \ref{T-----1}): Let $z_0:=z$, then for every $i\geq 0$, let $k_i:=\max\{k\in \N: Y_k<z_i\}$ and let $y_i:=Y_{k_i}$, and let $z_{i+1}:= Z_{k_i}$. Finally note that $d_{\mathcal N}(x,[0,2^n])= T:=\inf\{t\in \N, z_{t}\leq 2^n\}$.

\begin{figure}[!h] 
\centering
\includegraphics[scale=0.65]{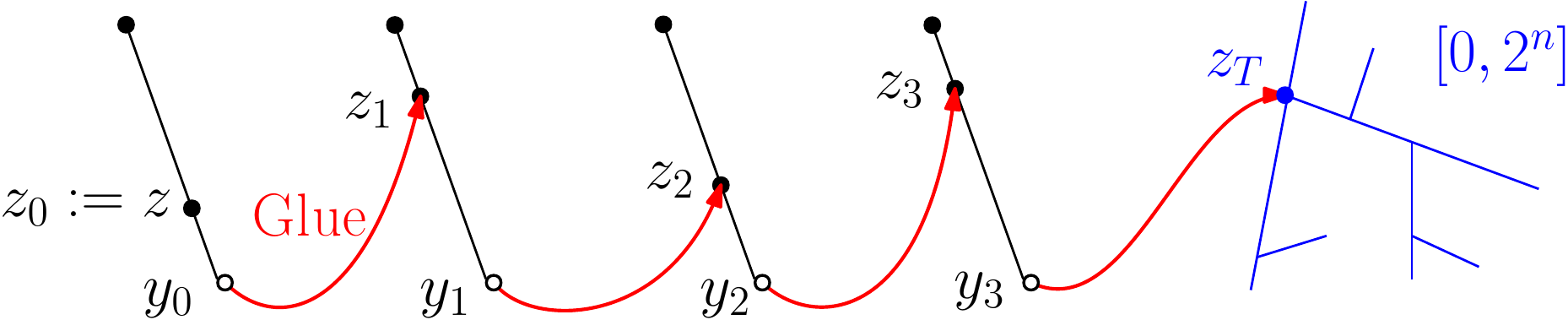}
\caption{A typical construction of $(z_i,y_i)_{i\in \N}$. Each segment represents a branch.}
 \label{T-----1}
\end{figure}

Then for every $i\geq 0$ let $\mathcal F_i:=\sigma(\mathcal F,z_0,\dots, z_i)$. Recall that conditionally on $\mathcal F$, $(Z_i)_{i\in \N}$ are independent so $((y_i,z_i),\mathcal F_i)_{i\geq 0}$ is a Markov chain. Also $T$ is a stopping time for this Markov chain. Moreover by definition of $(Z_i)_{i\in \N}$, if $z_i\geq 2^n$,
\[ \proba(z_{i+1}\leq 2^n | \mathcal F_i)=\frac{\mu[0,2^n]}{\mu[0,y_i]}\geq \frac{\mu[0,2^n]}{\mu[0,z]} \geq \frac{\mu[0,2^n]}{\mu[0,2^{n+1}]}\geq 1/3, \]
where the last inequality holds a.s. for every $n$ large enough by Lemma \ref{Recall icrt} (a) (b).
Thus, 
\[ \proba(T\geq 4n)=\proba(z_{4n-1}>2^n)\leq (2/3)^{4n-1}.\]

Therefore by an union bound,
\[ \proba \left (\left . \exists 1\leq i \leq 2^{2n+3}: Y_i\leq 2^{n}, d_{\mathfrak N}(Y_{i},[0,2^{n+1}])\geq 4n  \right | \mathcal F \right )=O(1/n^2).\]
Hence, by the Borel--Cantelli Lemma, and by Lemma \ref{Recall icrt} (c), a.s. for every $n$ large enough,
\[ \max\{d_{\mathfrak N}(Y_{i},[0,2^n]): i\in \N, Y_i\leq 2^{n+1}\}\leq  4n-1.\]
Finally simply note that $d_{\mathfrak N}(\cdot,[0,2^n] )$ is constant among branches. 
\end{proof}
\section{Basic properties on $\L$.} \label{Compacito}
\subsection{Proof of the compactness of $\L$.} \label{CompactLooptreeSec}
We first show an upper bound on $d_\L$. Then we introduce the projections on $\L_l:=[0,l]\times[0,1]$. Then we show that for every $l\in \R^+$, $(\L_l,d_\L)$ is a.s. compact. Then we upper bound for $n\in \N$, $d_H(\L_{2^n},\L_{2^{n+1}})$. Finally, we show that $\L$ is a.s. compact with the chaining method.

\begin{lemma} \label{BAKA} Almost surely for all $(x,u),(y,v) \in \T\times [0,1]$, we have $d_\L((x,u),(y,v))\leq \mu\llbracket x,y\rrbracket $.
Also for every $x\in \T$, $(y,v) \in \T\times [0,1]$, we have $d_\L((x,U_{x,y,v}),(y,v))\leq \mu\rrbracket x,y\rrbracket $.
\end{lemma}
\begin{proof} We focus on the second assertion as the first can be proved similarly. Let $u=U_{x,y,v}$. Recall that by definition, writing $\dc$ for the distance on the torus $[0,1]$,
\[ d_{\L}((x,u),(y,v))=\frac{\theta^2_0}{4}d_\T(x,y)+ \sum_{i\in \N}\theta_i \dc(U_{X_i,x,u},U_{X_i,y,v}).
\]
Then note that for every $i\in \N$ such that $X_i\notin \llbracket x,y\rrbracket$, $x,y$ are connected in $\T\backslash \{X_i\}$, hence $U_{X_i,x,u}=U_{X_i,x}=U_{X_i,y}=U_{X_i,y,v}$. Also $U_{x,x,u}=u=U_{x,y,v}$. Therefore,
\[ d_{\L}((x,u),(y,v))\leq \theta^2_0d_\T(x,y)+ \sum_{i\in \N: X_i\in \rrbracket x,y\rrbracket }\theta_i= \mu\rrbracket x,y\rrbracket. \qedhere\]
\end{proof}
\begin{remark}
Note that Lemma \ref{Recall icrt} (c) (d) implies $\mu[Y_i,Y_{i+1}]=O(i^{-1/2+o(1)})$. Also by Lemma \ref{BAKA} the diameter of the branches of the looptree is bounded by $(\mu[Y_i,Y_{i+1}])_{i\in \N}$. Thus we  are close from the compact setting of Curien and Haas \cite{CurienHaas} and of S\'enizergues \cite{Seni}.
\end{remark}

We now introduce the projections on $\T$ and on $\L$ (see Figure \ref{looptreefig6}). For every $x\in \T$, and $l\in \R^+$, let $\pT_l$ be the {\it projection of $x$ on $\T_l:=[0,l]$}, that is the unique $z\in \T_l$ which minimizes $d_\T(x,z)$. 
\begin{figure}[!h] \label{looptreefig6}
\centering
\includegraphics[scale=0.7]{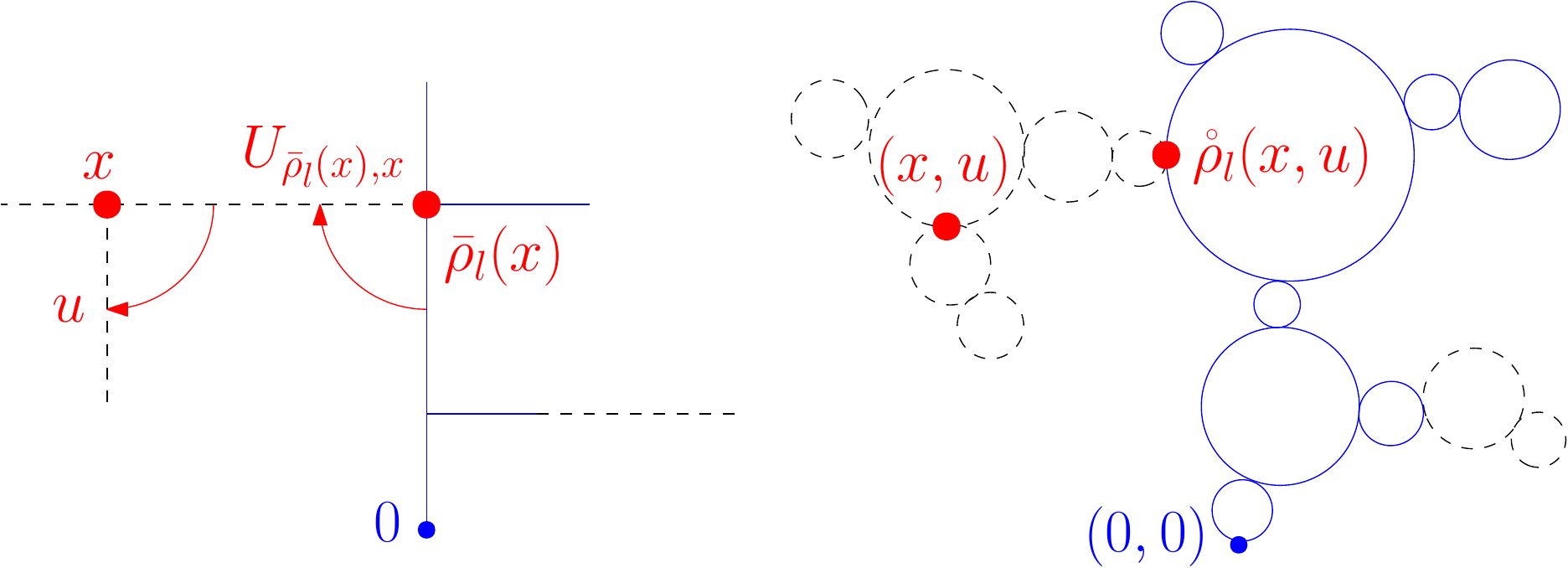}
\caption{Two representations of $\pL_l$, from the tree on the left, and from the looptree on the right. $\T_l$ and $\L_l$ are in blue. $\T\backslash \T_l$ and $\L\backslash \L_l$ are in black dashed. Here $U_{\pT_l(x),x,u}=U_{\pT_l(x),x}$ since $\pT_l(x)\neq x$.  
%
}  
\label{explore1} \label{looptreefig6}
\end{figure}%
\begin{lemma} \label{BAKA NO TEST} For every $(x,u)\in \T\times[0,1]$ and $l\in \R^+$, $\beta\in \L_l\mapsto d_\L((x,u),\y)$ reach its minimum at $\pL_l(x,u):=(\pT_l(x),U_{\pT_l(x),x,u})$. We call $\pL_l(x,u)$ the {\it projection of $x$ on $\L_l$}.
\end{lemma}
\begin{proof} Note that, since $\T$ is a $\R$-tree, for every $y\in [0,l]$, for every $z\in \rrbracket \pT_l(x),x\rrbracket$, $U_{z,y}=0$. Thus, for every $(y,u)\in \L_l$, since $d_\T(x,y)\geq d_\T(x,\pT_l(x))$,
\begin{equation} d_{\L}((x,u),(y,v))\geq  \frac{\theta^2_0}{4}d_\T(x,\pT_l(x))+\sum_{i:X_i\in \rrbracket \pT_l(x),x\rrbracket}\theta_i  \dc(U_{X_i,x,u})=d_\L(x,\pL_l(x)),  \label{4/3/18hb} \end{equation}
where the last inequality is obtained by a small modification of the proof of Lemma \ref{BAKA}.
\end{proof}
\begin{lemma} \label{CompactLoopFirst}Fix $l>0$. Almost surely, $d_\L$ is finite on $\L_l=[0,l]\times[0,1]$, and $(\L_l,d_l)$ is compact.
\end{lemma}
\begin{proof} The finiteness of $d_\L$ directly follows from Lemma \ref{BAKA} since for every $(x,u),(y,v)\in \L_l$, $\llbracket x,y\rrbracket\subset [0,l]$, and since almost surely $\mu$ is locally finite. So let us focus on the compactness. Let $((x_n,u_n))_{n\in \N}\in(\L_l,d_\L)^\N$. Up to extraction, we may assume that $(x_n)_{n\in \N}$ converges for $d_\T$ toward $x\in[0,l]$. Then note that, for every $X\neq x$ and for every $n$ large enough, $x_n$ and $x$ are connected in $\T\backslash \{X\}$, so $U_{X,x,u}=U_{X,x_n,v}$. Furthermore, up to extraction, we may also assume that $U_{x,x_n,u_n}\to u\in [0,1]$ as $n\to \infty$. Therefore, by dominated convergence, a.s. 
\[ d_{\L}((x_n,u_n),(x,u))=\frac{\theta^2_0}{4}d_\T(x_n,x)+ \sum_{i\in \N:X_i\leq l}\theta_i \dc(U_{X_i,x,u},U_{X_i,y,v})\to 0.  \qedhere
\]
\end{proof}

We now show that $(\L_{2^n})_{n\in \N}$ is a Cauchy sequence for $d_H$.
\begin{lemma} \label{ChainingStart} The following assertions hold a.s. for every $n$ large enough:
\begin{compactitem}
\item[a)] $\max_{x\in [0,2^{n+1}]}\min_{y\in [0,2^n]} \mu \rrbracket y,x\rrbracket \leq 4n^3/2^n$.
\item[b)] $d_H(\L_{2^n},\L_{2^{n+1}})\leq 4n^3/2^n$.
\end{compactitem}
\end{lemma}
\begin{proof} 
Let $(x,u)\in \L_{2^{n+1}}$. By lemma \ref{BAKA}, 
\[ d_\L((x,u),\pL_{2^n}(x,u))\leq \mu \rrbracket \pT_{2^n}(x),x\rrbracket. \]
 Also, a.s. for every $n$ large enough, by Lemma \ref{CountingLemma} $\rrbracket  \pT_{2^n}(x),x\rrbracket$ morally consists in the union of a part of at most $4n$ branches of the form $]Y_i,Y_{i+1}]$, which have by Lemma \ref{Recall icrt} mass at most $n^2/2^n$. Hence $\mu\rrbracket \pT_{2^n}(x),x\rrbracket\leq 4n^3/2^n$.
  Since $(x,u)\in \L_{2^{n+1}}$ is arbitrary, this concludes the proof.
\end{proof}
We will reuse the following intermediate result in \cite{ExcursionStick}:
\begin{lemma} \label{MuBounded} Almost surely $x\mapsto \mu\llbracket 0,x\rrbracket $ is bounded on $\T$. 
\end{lemma}
\begin{proof} First by Lemma \ref{ChainingStart} (a) almost surely, for every $n$ large enough,
\[ \max_{x\in [0,2^{n+1}]} \mu\llbracket 0,x\rrbracket \leq \max_{x\in [0,2^{n}]} \mu\llbracket 0,x\rrbracket +4n^3/2^n. \]
Thus since $\sum 4n^3/2^n<\infty$, $x\mapsto \mu\llbracket 0,x\rrbracket $ is bounded on $\R^+$. Moreover for every $x\in \T\backslash \R^+$, by monotone convergence, and since $\mu(\T\backslash \R^+)=0$, as $y\to x$, $y\in \llbracket 0,x\llbracket$,
\begin{equation} \mu\llbracket 0,y\rrbracket \limit \mu\llbracket 0,x\llbracket = \mu\llbracket 0,x\rrbracket. \qedhere \label{3/3/14h} \end{equation}
\end{proof}
\begin{proof}[Proof of Theorem \ref{CompactLooptree}] By Lemma \ref{MuBounded} and \eqref{3/3/14h}, for every $x\in \T$, $\min_{y\in \R^+}\mu\llbracket y,x\rrbracket=0$. Hence, by Lemma \ref{BAKA}, $d_H(\T\times[0,1],\R^+\times[0,1])=0$. So, since $\L$ is the completion of $(\T\times[0,1],d_\L)$, we have $d_H(\L,\R^+\times[0,1])=0$. 

Then by Lemma \ref{ChainingStart}, $(\L_{2^n})_{n\in \N}$ is a Cauchy sequence for the Hausdorff pseudo-distance. Also, by Lemma \ref{CompactLoopFirst}, for every $n\in \N$, $\L_{2^n}$ is compact. Hence, since $\L$ is complete, $(\L_{2^n})_{n\in \N}$ converges toward a compact subset of $\L$. This subset contains $\R^+\times[0,1]$. And since $d_H(\L,\R^+\times[0,1])=0$, it is at distance $0$ of $\L$. Therefore, $\L$ is compact.
\end{proof}
In passing note that we also prove the following result, which we will reuse in \cite{Dsnake}.
\begin{proposition} \label{FBLooptreeThm}  Almost surely as $l\to \infty$, $d_H(\L_l,\L)\to 0$.
\end{proposition}
\subsection{Proof of Proposition \ref{Completion}}\label{CompletionSec}
Toward (a), simply recall that  $d_H(\L,\R^+\times[0,1])=0$.

Toward (b), to show that $\T\times[0,1]=\L$ it is enough to prove that $(\T\times[0,1],d_\L)$ is compact. Let $((x_n,u_n))_{n\in \N}$ be a sequence in $\T\times[0,1]$. Since $\T$ is compact, we may assume up to extraction that 
$x_n\to x\in \T$, and $U_{x,x_n,u_n}\to u\in [0,1]$. Let us prove $(x_n,u_n)\to (x,u)$. There is two cases:

Either $x\notin \R^+$. In this case, for every $l\in \R^+$, for every $n$ large enough, $x$ and $x_n$ are connected in $\T\backslash [0,l]$. And it follows, by Lemma \ref{BAKA NO TEST}, that for every $n$ large enough,
\[ d_\L((x_n,u_n),(x,u))\leq 2d_H([0,l]\times[0,1],\L), \]
 which converges a.s. to $0$ as $l\to\infty$, by Proposition \ref{FBLooptreeThm}.
 
 Or $x\in \R^+$. In this case, for every $l\geq x$, as $n\to \infty$, by Lemma \ref{BAKA NO TEST}, $\pL_l(x_n,u_n)\to(x,u)$. Thus a.s.  
 \[ \limsup_{n\to\infty} d_\L((x_n,u_n),(x,u))\leq d_H(\L_l,\L)\limit_{l\to \infty} 0. \]

Then we show that $p_{\T,\L}:(x,u)\in \T\times[0,1]\mapsto x$ is continuous on $\T\times[0,1]$. 
 We argue by contradiction. Assume that there exists $(x_n,u_n)\in \L^\N$ such that as $n\to \infty$, $(x_n,u_n)\to (x,u)\in \L$ but $x_n\nrightarrow x$. Since $\T$ is compact we may assume up to extraction that $x_n\to y\in \T$.

Since $x\neq y$, there exists $a<b\in \R^+\backslash \{x,y\}$ such that $[a,b]\subset \rrbracket x,y\llbracket$. Then by density of $\mu$ on $\R^+$ (see proof of Theorem 3.1 \cite{ICRT1}), either $\theta_0>0$, so 
\[ \liminf_{n\in \N} d_\L((x_n,u_n),(x,u))\geq \frac{\theta^2_0}{4}d_\T(y,x)>0, \] 
or there exists $i\in \N$ with $\theta_i>0$ and $X_i\in [a,b]\backslash \{x\wedge y\}$, so for every $n\in \N$ large enough, 
\[d_\L((x_n,u_n),(x,u))\geq \theta_i U_{X_i}>0, \]
where the last inequality holds almost surely. Both cases contradict $d_\L((x_n,u_n),(x,u))\to 0$.


Toward (c), note that by definition of $d_\L$, for every $i\in \N$ with $\theta_i>0$, $\x\mapsto U_{X_i,\x}$ is $1/\theta_i$ Lipschitz from $\T\times[0,1]$ to $([0,1],\dc)$, and so extends by continuity on $\L$. 

Toward (d), by (c) and definition of $d_\L$ on $\T\times[0,1]$, it is enough to show that as $n\to \infty$, $\alpha,\beta\in \L^2\mapsto  \sum_{i=1}^n\theta_i \dc(U_{X_i,\x},U_{X_i,\y})$ converges uniformly  as $n\to \infty$.
To this end, first note that by dominated convergence, for every $l\in \R^+$, as $n,m\to \infty$, a.s. (see proof of Lemma \ref{BAKA})
\begin{equation*}  \Delta_{n,m,l}:=\max_{\x,\y\in \L_l}\sum_{i=n}^m\theta_i \dc(U_{X_i,\x},U_{X_i,\y})\leq \sum_{i=n}^\infty \theta_i\1_{X_i\in [0,l]}\to  0. \end{equation*}
Then, by definition of $d_\L$, and the triangular inequality, for every $l\in \R^+$, as $n,m\to \infty$, a.s.
\begin{align} \Delta_{n,m} & :=\max_{\x,\y\in \T\times [0,1]}\sum_{i=n}^m\theta_i \dc(U_{X_i,\x},U_{X_i,\y}) \notag
\\ & \leq 2d_H(\L_l,\L)+\max_{\x,\y\in \T\times [0,1]}\sum_{i=n}^m\theta_i \dc(U_{X_i,\pL_l(\x)},U_{X_i,\pL_l(\y)})  \notag
\\ & \leq 2d_H(\L_l,\L)+\Delta_{n,m,l} \notag
\\ & \to 2d_H(\L_l,\L). \label{23/02/21h} \end{align}
By Proposition \ref{FBLooptreeThm} a.s. $d_H(\L_l,\L)\to 0$. Thus by \eqref{23/02/21h} as $n,m\to\infty$ a.s. $\Delta_{n,m}\to 0$. The maximum in \eqref{23/02/21h} is then directly extended to $\L$ by (c). This yields the desired uniform convergence. 
\section{Preliminary results on left, front, right} \label{OLAGOURDE}
\label{LeftSec}
In this section is we prove the technical results necessary to construct and study $\Cont$ in Section \ref{ContourPathSec}. In Section \ref{PrelimLeftSec}, \ref{Generic} we prove generic results on $\Rgt,\too$, and on $\p_{\Lft},\p_{\Fnt},\p_{\Rgt}$. In Section \ref{ExtendProjSec}, \ref{ExtendLeftSec} we extend respectively $(\pL_l)_{l>0}$ and $\p_{\Lft}$ to $\L$. In Section \ref{PrelimContSec} we estimate $\p_{l,\Lft}$.
\subsection{Properties of $\Rgt$, $\too$.} \label{PrelimLeftSec}
Recall Section \ref{DefLeftFrontRight}. 
 The following lemma is proved in Appendix \ref{SadoMasoSec}:
\begin{lemma} \label{prelimOrder} $\too$ is a partial order. $\Rgt$ is a strict partial order. 
$\too$ raises $\prec$ (see Figure \ref{looptreefig5}): for every $\y\Rgt \z\in \T\times[0,1]$, for every $\x\in \T\times[0,1]$, $\z\too \x\ply \y \Rgt \x$ and $\y \too \x\ply \x \Rgt \z$. 
\end{lemma}
\begin{figure}[!h] \label{looptreefig5}
\centering
\includegraphics[scale=0.9]{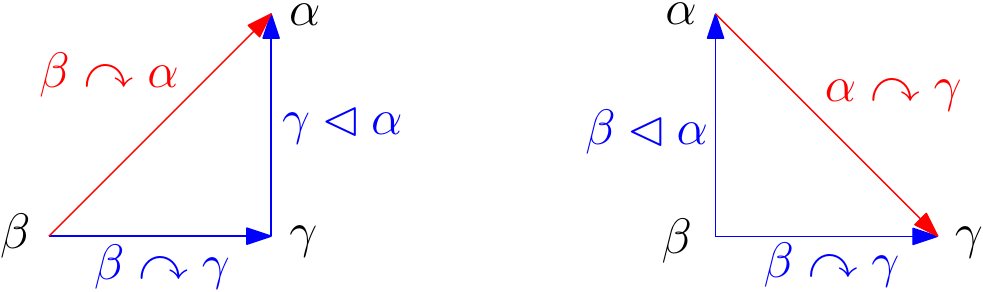}
\caption{The cases of Lemma \ref{prelimOrder}. The assumptions are blue. The conclusions are red. }  
\label{explore1} \label{looptreefig5}
\end{figure}%
\begin{lemma} \label{KeepProjection} Recall $\pT, \pL$ from section \ref{CompactLooptreeSec}.  For every $l>0$ and $\x=(x,a),\y=(y,b) \in \T\times[0,1]$: 
\begin{compactitem}
\item[(a)] If $\pL_l(\x)\Rgt \pL_l(\y)$ then $\x\Rgt \y$.
\item[(b)] If $\pL_l(\x)\too \pL_l(\y)$ and $\pL_l(\x)\neq \pL(\y)$ then $\x=\pL_l(\x)\too \y$.
\end{compactitem}
\end{lemma}
\begin{proof} First by definition of $\pL$, $\pL_l(\x)=(\pT_l(x),U_{\pT_l(x),\x})\too \x$ and similarly $\pL_l(\beta)\too \beta$. (a) follows since $\too$ raises $\Rgt$. 
Toward (b), $\pL_l(\x)\too \x$, and $\pL_l(\x)\too \pL_l(\y)$, so either: 
\begin{compactitem}
\item $x=\pT_l(x)$. 
\item $\pT_l(y)=\pT_l(x)$ and then by $\pL_l(\x)\too \pL_l(\y)$, we have $\pL_l(\x)=\pL_l(\y)$.
\item $x,\pT_l(y)$ are connected in $\T\backslash \{\pT_l(x)\}$ which is absurd by definition of $\pT$.
\end{compactitem}
Hence, $x=\pT_l(x)$. So since $\pL_l(\x)\too \x$, by definition of $\too$, $\x=\pL_l(\x)$. Finally $\x\too \pL_l(\y)\too \y$.
\end{proof}
\subsection{Generic properties on $\p_{\Lft}, \p_{\Fnt}, \p_{\Rgt} $} \label{Generic}
In this section $(\T,d,\rho,u)$ denotes a plane $\R$-tree, and $\p$ denotes a probability Borel measure on $\T$. 
\begin{lemma} The following assertions hold: \label{prelimLeft} 
\begin{compactitem} 
\item[(a)] For every $\x\in \T\times[0,1]$, $\p_{\Lft}(\x)+\p_{\Fnt}(\x)+\p_{\Rgt}(\x)=1$.
\item[(b)]$\p_{\Lft}$ is increasing for $\prec$.
\item[(c)] If $A$ is a random variable with law $\p\times\1_{l\in [0,1]}dl$, then almost surely $\p_\Fnt(A)=0$.
\item[(d)] With the same notations, $\p_\Lft(A)$ is uniform in $[0,1]$. 
\item[(e)] $\{\p_\Lft(\x)\}_{\x\in \T\times[0,1]}$ is dense in $[0,1]$. 
\end{compactitem}
\end{lemma}
\begin{proof} Toward (a), note that for every $\x\in \T\times[0,1]$, we have the following partition,
\begin{equation} \T\times[0,1]=\{\y: \y\neq \x, \y\too \x\}\cup \{\y:\x\Rgt \y\}\cup\{\y:\x\too \y\}\cup \{\y:\y\Rgt \x\}. \label{25/02/0h} \end{equation}
Then by Fubini's theorem since for every $y\in \T$, $\#\{v\in [0,1], (y,v)\too \x\}\leq 1$, we have
$\mu_{\L}\{\y: \y\too \x\}=0.$ (a) follows from \eqref{25/02/0h}.

Toward (b), by Lemma \ref{prelimOrder}, for every $\x\prec \y\in \L$, $\{\z: \z\Rgt \x\}\subset \{\z: \z\Rgt \y\}$.

Toward (c), let $A,A'$ be independent random variables with law $\mu_\L$. Since $\T$ is separable, and balanced, the events $A\Rgt A', A\too A'$ are measurable (see Lemma \ref{Measurable2}). 
The measurability of the other random variables in the proof is then due to Fubini's Theorem. 

Then note that a.s. $\p_{\Fnt}(A)=\proba[A\too A'|A]$, and by (a) a.s. $\p\{\x: \x\too A'\}=\proba[A\too A'|A']=0$. Hence, by Fubini's theorem, a.s. $\p_{\Fnt}(A)=0$.

Toward (d), let $(A_i)_{i\in \N}$ be a family of independent random variables with law $\p\times\1_{l\in [0,1]}dl$. By the weak law of large number a.s.
\begin{equation} \frac{1}{n}\sum_{i=1}^n\proba(A_i\Rgt A_1|A_1)\limit_{n\to \infty} \p_{\Lft}(A_1). \label{25/02/1h}\end{equation}
Furthermore, by (b) for every $i\neq j\in \N$,  
\[ \proba(A_i\too A_j)=\E[\proba(A_i\too A_j|A_i)]=\E[\p_\Fnt(A_i)]=0. \] Hence, since $\prec$ is a total order, a.s. for every $i,j\in \N$ either $A_i\Rgt A_j$ or $A_j\Rgt A_i$. Moreover, the law of $(A_i)_{i\in \N}$ is invariant by permutations. Hence for every $n\in \N$, $\#\{i\leq n: A_i\Rgt A_1\}$ is uniform in $\{0,1,2,\dots, n-1\}$. Finally (d) follows from \eqref{25/02/1h}. 
(e) directly follows from (d).
\end{proof}

\subsection{Extension of $\pL_l$ to $\L$} \label{ExtendProjSec}
Recall notations $\pT, \pL$ from section \ref{CompactLooptreeSec}. Although for all $l>0$, $\pL_l$ is continuous on $\T\times[0,1]$, we do not want to extend it by continuity. Indeed, a continuous extension is only defined up to $\sim_\L$, while $\Rgt$, $\too$ are not. So we prove instead that 
$\pL_l$ is locally constant around each vertex of $\L\backslash (\T\times[0,1])$, and extend it naturally. First, we prove that $\pL$ is piecewise constant: 

\begin{lemma} \label{Extend1} For every $l>0$, for every $C$ connected component of $\T\backslash [0,l]$, $\pT$ is constant on $C$, and $\pL$ is constant on $C\times[0,1]$.
\end{lemma}
\begin{proof} The first assertion is classic, so we leave it to the reader as an exercise. Let $x,y\in C$ and let $a,b\in[0,1]$. $\pT(x)=\pT(y)$.
%
Since $x,y\neq \pT(x)$, $U_{\pT_l(x),x,a}=U_{\pT_l(x),x}$ and $U_{\pT_l(x),y,b}=U_{\pT_l(x),y}$. Also since $x,y$ are connected in $\T\backslash \{\pT(x)\}$, we have $U_{\pT_l(x),x}=U_{\pT_l(x),y}$. So $U_{\pT_l(x),x,a}=U_{\pT_l(y),y,b}$.
Therefore, $\pL(x,a)=\pL(y,b)$. 
\end{proof}
For every $\x\in \L$, and $\e>0$ let $B(\x,\e)$ denote the open ball for $d_\L$ of center $\x$ and radius $\e$.
\begin{lemma} \label{LocalConstant} For every $l>0$, for every $\x\in \L$, $\pL_l$ is constant on $(\T\times [0,1])\cap B(\x,d_\L(\x,[0,l])/2)$.
\end{lemma}
\begin{proof}  We argue by contradiction. Let $l>0$, $\x\in \L$, $\y=(y,b),\z=(z,c)\in \T\times[0,1]$. Assume that $\pL_l(\y)\neq \pL_l(\z)$, and $d_\L(\y,\x)<d_\L(\x,[0,l])/2$, and $d_\L(\z,\x)<d_\L(\x,[0,l])/2$. Let $D=d_\L(\x,[0,l])$. By the triangular inequality, 
\begin{equation} d(\y,[0,l])>D/2 \quad ; \quad d(\z,[0,l])>D/2 \quad ; \quad d_\L(\y,\z)< D. \label{4/3/18hd} \end{equation}

Moreover, since $\pL_l(\y)\neq \pL_l(\z)$, by Lemma \ref{Extend1}, $y,z$ are disconnected in $\T\backslash[0,l]$. Hence,
$\llbracket y, \pT_l(y)\llbracket \cap  \rrbracket \pT_l(z),z\llbracket=\emptyset.$ 
As a result, writing
\begin{equation} S_\y:= \frac{\theta^2_0}{4}d(y, \pT_l(y))+ \sum_{i:X_i\in \llbracket y, \pT_l(y)\llbracket } \theta_i\dc(U_{X_i,\y},U_{X_i,\z}), \label{10/3/9h} \end{equation}
and similarly for $S_\z$, we have 
$d_\L(\y,\z)\geq S_\y+S_\z.$ 

Also,  for every $X_i\in \llbracket y, \pT_l(y)\llbracket$, $\pT_l(y)$ and $z$ are connected in $\T\backslash \{X_i\}$ so $U_{X_i,\z}=U_{X_i,\pL_l(\y)}$. Hence, by \eqref{10/3/9h}, \eqref{4/3/18hb}, $S_\y=d_\L(\y,\pL_l(\y))$.
and similarly for $S_\z$. Therefore, since $d_\L(\y,\z)\geq S_\y+S_\z$, \[d_\L(\y,\z)\geq d_\L(\y,\pL_l(\y))+d_\L(\z,\pL_l(\z)).\]
This contradicts \eqref{4/3/18hd}.
\end{proof}
\begin{lemma} \label{ComposProj} For every $0<r\leq s$, $\pL_{r}\circ \pL_{s}=\pL_{s}\circ \pL_{r}=\pL_{r}$. 
\end{lemma}
\begin{proof} First since $\pL_{r}$ have value in $\L_r\subset \L_s$ and $\pL_{s}$ is the identity on $\L_s$, we have $\pL_{s}\circ \pL_{r}=\pL_{r}$. Then, since $\pL_r$, and $\pL_s$ are locally constant around each vertex of $\L\backslash \T\times[0,1]$, and since $\T\times[0,1]$ is dense, it is enough to show that $\pL_{r}\circ \pL_{s}=\pL_{r}$ on $\T\times[0,1]$.

Let $(x,a)\in \T\times[0,1]$. By definition of $\pT$, $\pT_{r}\circ \pT_s(x)$ is the vertex $z$ in $\llbracket 0,\pT_s \rrbracket \cap [0,r]$ which maximizes $d_\T(0,z)$. Moreover, since $\T$ is a $\R$-tree, $\rrbracket \pT_s,x\rrbracket \subset \T\backslash [0,s]$. Hence, since $r\leq s$, $\pT_{r}\circ \pT_s(x)$ is the vertex $z$ in $\llbracket 0,x \rrbracket \cap [0,r]$ which maximizes $d_\T(0,z)$, which is $\pT_r(x)$. 
Finally by definition of $\pL$, we have $\pL_r\circ \pL_s(x,a)\too \pL_s(x,a)\too (x,a)$. So, 
\[ \pL_r\circ \pL_s(x,a)=(\pT_r(x),U_{\pT_r(x),x,a})=\pL_r(x,a). \qedhere \]
\end{proof}
\subsection{Continuous extension of $\p_\Lft$ to $\L$} \label{ExtendLeftSec}
\begin{lemma} \label{ExtendPleft} Recall the notation $\p_\L:=\p\times\1_{x\in\mu[0,1]}dx$. The map
\begin{equation} \x\in \L\mapsto \p_\L\left ( \bigcup_{l>0} \{\y\in \T\times[0,1], \pL_l(\y) \Rgt \pL_l(\x) \} \right ) \label{6/3/20h} \end{equation}
is well defined and coincide with $\p_\Lft$ on $\T\times[0,1]$.
\end{lemma}
\begin{remark} One may see each $\x\in \L\backslash (\T\times[0,1])$ as the "end" of the infinite branch (see Figure \ref{looptreefig7}) $\{p_{\T,\L}\circ \pL_l(\x),l\in \R^+\}$, and the relations of $\T\times[0,1]$ may be extended to those missing points. 
\end{remark}
\begin{figure}[!h] \label{looptreefig7}
\centering
\includegraphics[scale=0.65]{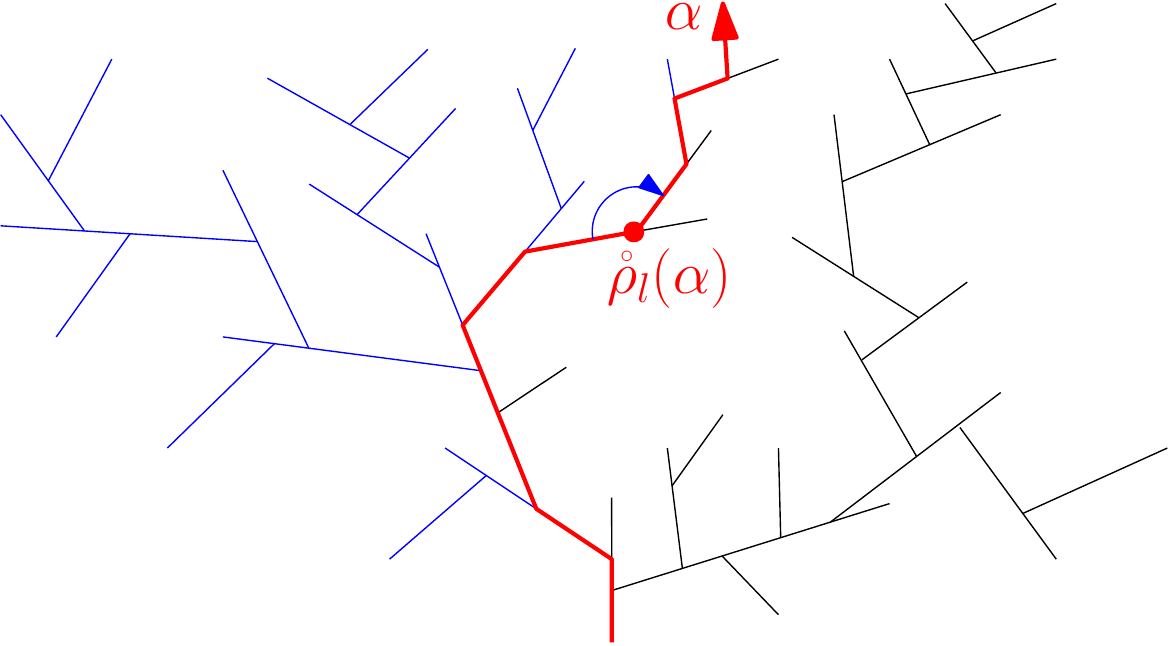}
\caption{ A simplified non compact ICRT,  with a spinal representation of $\alpha\in \L\backslash (\T\times[0,1])$. The infinite branch $\{p_{\T,\L}\circ \pL_l(\x),l\in \R^+\}$ is in red. \eqref{6/3/20h} estimates $\p$ of the blue part, on the left. 
} 
\label{explore1} \label{looptreefig7}
\end{figure}%
\begin{proof} First by Lemmas \ref{KeepProjection} and \ref{ComposProj}, for every $0<l\leq l'$ and $\y \in \T\times[0,1]$ such that $\pL_l(\y) \Rgt \pL_l(\x)$ we have $\pL_l\circ\pL_{l'}(\y)\Rgt \pL_l\circ\pL_{l'}(\x)$ so $\pL_{l'}(\y)\Rgt\pL_{l'}(\x)$. Thus $(\{\y\in \T\times[0,1], \pL_l(\y) \Rgt \pL_l(\x) \})_{l>0}$ is increasing. Furthermore, those sets are measurable (see Appendix \ref{SadoMasoSec}). Hence \eqref{6/3/20h} is well defined.

Next let $\x=(x,a),\y=(y,b)\in \T\times[0,1]$. If there exists $l>0$ such that $\pL_l(\y)\Rgt \pL_l(\x)$ then by Lemma \ref{KeepProjection}, $\y\Rgt \x$. Reciprocally, assume that $\y\Rgt \x$ and either $x\in \R^+$ or $x\neq y$. Then $y\wedge x\in \R^+$, $\pT_{x\wedge y}(x)=x$ and $\pT_{x\wedge y}(y)=y$. Also, since $\y\Rgt \x$, we have $U_{x\wedge y,y,b}<U_{x\wedge y,x,a}$. So,
\begin{equation} \pL_{x\wedge y}(y,b)=(x\wedge y,U_{x\wedge y,y,b})\Rgt (x\wedge y,U_{x\wedge y,x,a})=\pL_{x\wedge y}(x,a).\label{6/3/23h} \end{equation}

As a result, for every $(x,a)\in \T\times[0,1]$,
\[ \{\y, \y\Rgt (x,a)\} \subset \bigcup_{l>0} \{\y, \pL_l(\y) \Rgt \pL_l(x,a) \}\subset \{\y, \y\Rgt (x,a)\}\cup \{x\}\times[0,1], \]
and the first inclusion is an equality when $x\in \R^+$. Finally for every $x\in \T\backslash \R^+$, $\p(x)=0$. (Indeed, when $\theta_0=0$, $\sum \theta_i<\infty$, $\p$ have support $\{X_i\}_{i\in \N}\subset \R^+$. In the other case see \cite{ICRT1} Theorem 3.1.) So $\p_\Lft$ coincide with \eqref{6/3/20h}.
\end{proof}
\begin{lemma} \label{YOLONUL} Let $\p_\Lft$ denote \eqref{6/3/20h}. $\p_\Lft$ is continuous around each vertex of $\L\backslash \{\R^+\times [0,1]\}$.
\end{lemma}
\begin{proof} Fix $\x\in \L\backslash \{\R^+\times [0,1]\}$. First, recall that the sets of \eqref{6/3/20h} are increasing so, as $l\to \infty$,
\[ \p_\L(\{\y\in \T\times[0,1], \pL_l(\y) \Rgt \pL_l(\x) \})\to \p_\Lft(\x).\]
Furthermore, recall that by Lemma \ref{LocalConstant} $\pL_l$ is locally constant, hence for every $l>0$,
\[ \liminf_{\y\to \x} \p_{\Lft}(\y)\geq \p_\L(\{\y\in \T\times[0,1], \pL_l(\y) \Rgt \pL_l(\x) \})\to_{l\to \infty}  \p_{\Lft}(\x).\]

Let us show the other inequality. For every $\y\in \L$ let 
\[ S_\y:=\bigcup_{l>0} \{\z\in \T\times[0,1], \pL_l(\z) \Rgt \pL_l(\y) \}. \] 
Let $l>0$. Let $\y\in \L$ such that $\pL_l(\y)=\pL_l(\x)$. Let $(z,c)\in S_\y$. By Lemma \ref{KeepProjection} there exists $r>l$ such that $\pL_r(z,c)\Rgt \pL_r(\y)$. Let $(y,b)=\pL_r(\y)$. By Lemma \ref{ComposProj}, $\pL_l(y,b)=\pL_l(\y)=\pL_l(\x)$.

 If $y\wedge z>\pT_l(y)$, then $y$ and $z$ are connected in $\T\backslash \{\pT_l(y)\}$. So $\pT_l(y)\in \llbracket 0, z\rrbracket$ and $U_{\pT_l(y),y,b}=U_{\pT_l(y),z,c}$. Also, $\pL_l(y,b)=(\pT_l(y),U_{\pT_l,y,b})$. Hence, 
$\pL_l(\x)=\pL_l(y,b)\too (z,c)$. 

If $y\wedge z\leq \pT_l(y)$, then by Lemma \ref{ComposProj},
\[\pL_{y\wedge z}(y,b)=\pL_{y\wedge z}\circ \pL_l(y,b)=\pL_{y\wedge z}\circ \pL_l(\x)=\pL_{y\wedge z}(\x). \]
Moreover, we have $\pL_r(z,c)\Rgt (y,b)$ so by Lemma \ref{prelimLeft} $(z,c)\Rgt (y,b)$. Then by \eqref{6/3/23h}, we have $\pL_{y\wedge z}(z,c)\Rgt \pL_{y\wedge z}(y,b)$.
Hence, $\pL_{y\wedge z}(z,c)\Rgt \pL_{y\wedge z}(\x)$. And $(z,c)\in S_\x$.

To sum up, for all $l>0$, and $\y\in \L$ such that $\pL_l(\y)=\pL_l(\x)$, 
\begin{equation*} S_\y \subset S_\x \cup \{(z,c)\in \T\times[0,1], \pL_l(\x)\too (z,c)\}.\label{6/3/23hb} \end{equation*}
Let $I_l$ denote the right most set above. By Lemma \ref{LocalConstant}, to show that $\limsup_{\y\to \x} \p_{\Lft}(\y)\leq \p_{\Lft}(\x)$, it suffices to show that as $l\to \infty$, $\p_\L(I_l)\to 0$.  First since $\too$ is an order, and since for every $l\leq l'$, $\pL_l(\x)\too \pL_{l'}(\x)$, $(I_l)_{l>0}$ is decreasing. Then for every $z\in \T$, $\pT_l\to z$ as $l\to\infty$. So for every $(y,b)\neq (z,c)\in \T\times[0,1]$ such that $y\neq z$ or $z\in \R^+$, for every $l$ large enough $\pL_l(y,b)\neq \pL_l(z,c)$. Thus, $\cap_{l>0} I_l$ is included in a set of the form $\{z\}\times[0,1]$ with $z\in \T\backslash \R^+$, or of the form $\{(z,c)\}$. Finally recall that for every $z\in \T\backslash \R^+$, $\p(z)=0$. Hence, $\p_\L(\bigcap_{l>0}I_l)=0$. This shows the other desired inequality and thus concludes the proof.
\end{proof}
\subsection{Some preliminary results on $\p_{l,\Lft}$} \label{PrelimContSec}
For every $\nu$, $\sigma$-finite Borel measure on $\T$ and $\x,\y\in \T\times[0,1]$, let $\nu_{\Lft}(\x,\y)$ denote $\nu_{\Lft}(\y)-\nu_{\Lft}(\x)$.
\begin{lemma} \label{BorneCrucial} Almost surely for every $l>0$, $\x, \y\in \L_l=[0,l]\times[0,1]$, $|\mu_{l,\Lft}(\x,\y)|\geq d_\L(\x,\y)$. (See Figure \ref{looptreefig4}.)
\end{lemma}
 \begin{figure}[!h] \label{looptreefig4}
\centering
\includegraphics[scale=0.45]{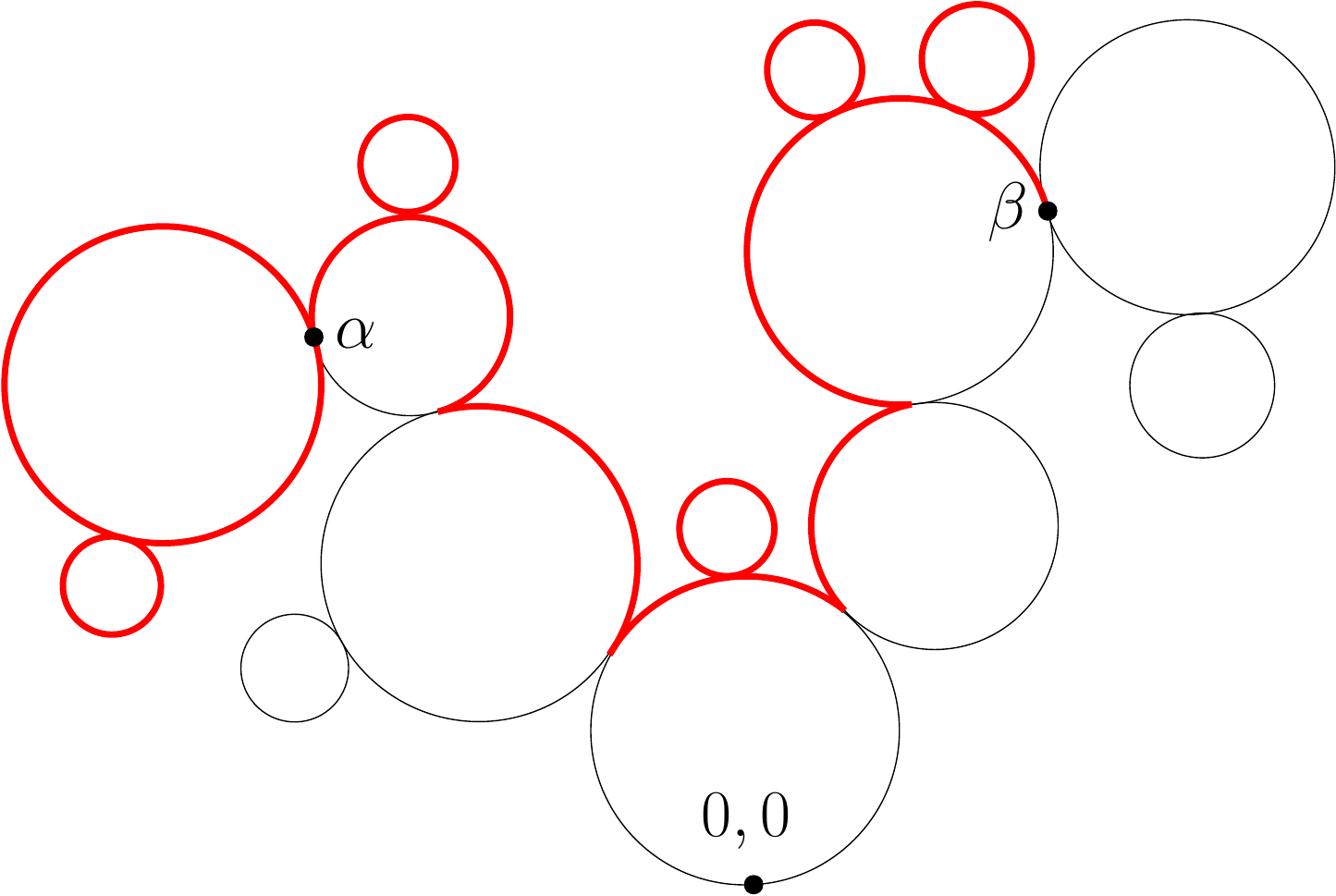}
\caption{An informal proof of Lemma \ref{BorneCrucial}: $\x\prec \y$. The set $S=\{\z, \z\Lft \y\}\backslash \{\z,\z\Lft \x\}$ is red. (It can be obtained by turning clockwise from $\x$ to $\y$.) Its total length is $\mu_{l,\Lft}(\x,\y)$. Note that this set contains a path between $\x$ and $\y$. This path have length at least $d_\L(\x,\y)$. 
}
\label{explore1} \label{looptreefig4}
\end{figure}
\begin{proof} Let $\x=(x,a),\y=(y,b)\in [0,l]\times[0,1]$. By symmetry we may assume $(x,a)\prec (y,b)$. By Lemma \ref{OSKOUR} (a), (b), $\{\z\in \L_l,\z\Lft \x\}\subset\{\z\in \L_l, \z\Lft \y\}$. Then writing 
\[S:=\{\z\in \L_l, \z\Lft \y\}\backslash \{\z\in \L_l,\z\Lft \x\}, \]
 we have $\mu_{l,\Lft}(\x,\y)=\mu(S)$.

For every $z\in \rrbracket x\wedge y, x\rrbracket$, $c\in (U_{z,x,a},1)$, we have $x\wedge z=z$, and $U_{z,x,a}<c=U_{z,z,c}$. So $(x,a)\Rgt(z,c)$. Also, $z$ and $x$ are connected in $\T\backslash \{x\wedge y\}$ so $z\wedge y=x\wedge y$, and $U_{z\wedge y,x,a}=U_{x\wedge y,x,a}<U_{x\wedge y,y,b}$. Hence, $(z,c)\Rgt (y,b)$. Therefore,
\[ S_1:=\{(z,c):z\in \rrbracket x\wedge y,x\rrbracket, c\in (U_{z,x,a},1) \}\subset S.\]
Similarly, 
\[ S_2:=\{(x\wedge y,c): c\in (U_{x\wedge y,x,a},U_{x\wedge y,y,b}) \}\subset S,\]
and,
\[ S_3:=\{(z,c):z\in \rrbracket x\wedge y,y\rrbracket, c\in (0,U_{z,y,b}) \}\subset S.\]

Moreover $S_1,S_2, S_3$ are disjoint. And writing $\dc$ for the distance on the torus $[0,1]$, since for every $z\in \rrbracket x\wedge y,x\rrbracket$, $U_{z,y,b}=0$,
\[ \mu(S_1)=\frac{\theta^2_0}{2} d_\T(x\wedge y,x)+\sum_{i:X_i\in \rrbracket x\wedge y,x\rrbracket} (1-U_{z,x,a})\geq \frac{\theta^2_0}{4} d_\T(x\wedge y,x)+\sum_{i:X_i\in \rrbracket x\wedge y,x\rrbracket}\theta_i \dc(U_{X_i,x,a},U_{X_i,y,b}). \]
Similarly, $\mu(S_2)\geq \mu\{x\wedge y\} \dc (U_{x\wedge y,x,u},U_{x\wedge y,y,b})$, and 
\[ \mu(S_3) \geq \frac{\theta^2_0}{4} d_\T(x\wedge y,y)+\sum_{i:X_i\in \rrbracket x\wedge y,y\rrbracket} \theta_i \dc(U_{X_i,x,a},U_{X_i,y,b}). \]

Finally, by sum, and since for every $i\in \N$, such that $X_i\notin \llbracket x,y\rrbracket$, $U_{X_i,x,a}=U_{X_i,y,b}$,
\[ \mu(S)\geq \mu(S_1)+\mu(S_2)+\mu(S_3)\geq \frac{\theta^2_0}{4}d_\T(x,y)+\sum_{i=1}^\infty \theta_i \dc(U_{X_i,x,a},U_{X_i,y,b})=d_\L(\x,\y). \qedhere \]

\end{proof}
\begin{lemma} \label{CVLeft} Almost surely for every $\x\in \T\times[0,1]$, as $l\to \infty$, $\p_{l,\Lft}(\x)\to \p_\Lft(\x)$.
\end{lemma}
\begin{proof} When $\theta_0=0$ and $\sum_{i=1}^\infty \theta_i<\infty$, we have $\mu(\R^+)<\infty$. As a result, $\mu_l\to \mu$ in total variation. So $\p_l\1_{x\in [0,1]}dx \to \p\1_{x\in [0,1]}dx$ in total variation. The desired result follows.

When $\theta_0\neq 0$ or $\sum_{i=1}^\infty\theta_i=\infty$, some extra care is needed since $(\p_l)_{l\in \R^+}$ only converges weakly. For every $(x,a)\in \T\times [0,1]$, let $S(x,a):=\{\y\in \T\times[0,1], \y\Rgt(x,a)\}$. Let 
\[S_1(x,a):=\{ z\in \T\backslash \llbracket 0, x\rrbracket, U_{x\wedge z,z}<U_{x\wedge z,x,a} \}.  \]
Note that $S_1(x,a)$ is a Borel set in $(\T,d_\T)$ as an union of connected component (see Appendix \ref{SadoMasoSec}). Moreover, by definition of $\Rgt$,
\[ S_1(x,a) \times[0,1]\subset S(x,a)\subset (S_1(x,a)\cup \llbracket 0,x\rrbracket )\times [0,1]. \]

Hence, for every $x,a\in \T\times[0,1]$
\begin{equation} \p_{l}(S_1(x,a))\leq \p_{l,\Lft}(x)\leq \p_{l}(S_1(x,a)\cup \llbracket 0,x\rrbracket),\label{sardineA} \end{equation}
and similarly,
\begin{equation} \p(S_1(x,a))\leq \p_{\Lft}(x)\leq \p(S_1(x,a)\cup \llbracket 0,x\rrbracket),\label{sardineD} \end{equation}
Moreover, note that $S_1(x,a)$ is an open set of $\T$ as a union of open connected components. Similarly, note that $S_1(x,a)\cup \llbracket 0,x\rrbracket$ is a closed set since its complementary is a union of open connected components. Furthermore, by \cite{ICRT1} Theorem 3.1, $\p(\R^+)=0$, and $\p$ has no atoms, so $\p(\llbracket 0,x\rrbracket)=0$. As a result, by Portmanteau Theorem, 
\begin{equation} \p_{l}(S_1(x,a))\to \p(S_1(x,a))  \quad \text{and} \quad \p_{l}(S_1(x,a)\cup \llbracket 0,x\rrbracket)\to \p(S_1(x,a)\cup\llbracket 0,x\rrbracket)=\p(S_1(x,a)). \label{sardineE} \end{equation}
The desired result follows by \eqref{sardineA},\eqref{sardineD}, \eqref{sardineE}.
\end{proof}
We now adapt  \cite{ICRT1} Lemma 5.1 to estimate precisely the evolution of $\p_{l,\Lft}$. 
\begin{lemma} \label{E=mc2}  Let $\mathbf U_X:= (U_{X,i})_{i\in \N}$. Almost surely $(\mu,\mathbf Y,\mathbf U_X)$ satisfies the following property. For all $a$ large enough, conditionally on $\mathcal F_a:=\sigma(\mu,\mathbf Y,(Z_i,U_{Z,i})_{i< a},\mathbf U_X)$, for every $\x,\y\in \L_{Y_a}$:
If $\mu_{l,\Lft}(\x,\y) \geq (\log^6 Y_a)/Y_a$ then with probability at least $1-1/Y_a^5$, for every $b\geq a$,
\[  \left (1-\frac 1 {\log Y_a} \right ) \p_{Y_a,\Lft} (\x,\y) \leq p_{Y_b,\Lft} (\x,\y) \leq  \left (1+\frac 1 {\log Y_a} \right) \p_{Y_a,\Lft}(\x,\y).\]
\end{lemma}
\begin{remark} With $\x=(0,0)$, for all $a\in \N$, since $\p_{Y_a,\Lft}(0,0)=0$, we have $\p_{Y_a,\Lft} (\x,\y)=\p_{Y_a,\Lft} (\y)$.
\end{remark}
\begin{proof} 
First by Lemma \ref{prelimLeft} (b), $\x\prec \y$. So, by Lemma \ref{OSKOUR} (a), (b), $\{\z: \z\Rgt \x\}\subset \{\z: \z\Rgt \y\}$. Then by construction for every $i\in \N$, conditionally on $\mathcal F_i$, $\z_i:=(Z_i,U_{Z,i})$ have law $\p_{Y_i}\times[0,1]$. 
Moreover, for every $i\geq a$, writing $I_i:=]Y_{i},Y_{i+1}]\times[0,1]$, conditionally on $\mathcal F_i$ a.s.
\begin{compactitem} 
\item With probability $\p_{Y_i,\Lft}(x)$, $\z_i\Rgt \x$. Then by Lemma \ref{OSGOUR}, for every $\delta \in I_i$, $\delta\Rgt \x$.  
\item With probability $\p_{Y_i,\Fnt}(x)$, $\x\too \z_i$. Then by Lemma \ref{OSKOUR} (d), for every $\delta \in I_i$, $\x\too \delta$. 
\item With probability $\p_{Y_i, \Rgt}(x)$, $\x\Rgt \z_i$. Then by Lemma \ref{OSKOUR} (b), for every $\delta \in I_i$, $\x \Rgt \delta$. 
\end{compactitem}
And similarly for $\y$. 

Therefore, by Lemma \ref{prelimLeft} (a), and $\{\z: \z\Rgt \x\}\subset \{\z: \z\Rgt \y\}$, for every $i\geq a$, a.s.
\[ \proba \left (\left . \mu_{Y_{i+1},\Lft}(\x,\y)=\mu_{Y_i,\Lft}(\x,\y)+\mu(Y_i,Y_{i+1}] \right | \mathcal F_i\right )=p_{Y_i,\Lft} (\x,\y), \]
 and,
\[ \proba \left (\left . \mu_{Y_{i+1},\Lft}(\x,\y)=\mu_{Y_i,\Lft}(\x,\y) \right | \mathcal F_i\right )=1-p_{Y_i,\Lft} (\x,\y). \]

As a result, $(\p_{Y_i,\Lft}(\x,\y),\mathcal F_i)_{i\geq a}$ is a P\'olya urn in the sense of \cite{ICRT1} Lemma A.1. And we can conclude exactly as in the proof of \cite{ICRT1} Lemma 5.1.
\end{proof}
\section{Construction and H\"older continuity of the contour path.} \label{ContourPathSec}
\subsection{Construction of the contour path $\Cont_l$ on $\L_l$.}
\begin{lemma} \label{First Cont}Recall that $\sim_\L$ denotes the metric equivalence on $(\L,d_\L)$. A.s. for every $l>0$ there exists a $\mu[0,l]$-Lipschitz function $\Cont_l:[0,1]\mapsto \L_l$ such that for every $\x\in \L_l$, $\Cont_l(\p_{l,\Lft}(\x)) \sim_\L \x$. 
\end{lemma}
\begin{remark} 
Since by Lemma \ref{prelimLeft} (e), $\{\p_\Lft(\x),\x\in \L_l\}$ is dense, $\Cont_l$ is unique up to $\sim_\L$. 
\end{remark}
\begin{proof} 
 A.s. for every $l>0$ the following holds:  Let $S_l:=\{\p_\Lft(\alpha),\alpha\in \L_l\}$. For every $u\in S_l$, we may chose $\Cont_l(u)\in \L_l$ such that $\p_\Lft(\Cont_l(u))=u$. Then by Lemma \ref{BorneCrucial}, for every $\x,\y \in \L_l$,
\begin{equation} d_\L(\x,\y)  \leq |\mu_\Lft(\x)-\mu_\Lft(\y)| = \mu[0,l] |\p_\Lft(\x)-\p_\Lft(\y)|. \label{HolderCont} \end{equation}
Thus for every $\x\in \L_l$, we have $d_\L(\x,\Cont_l(\p_{\Lft}(\x)))=0$. Also, by \eqref{HolderCont}, $\Cont_l$ is $\mu[0,l]$-Lipschitz on $S_l$. Furthermore by Lemma \ref{prelimLeft} (e), $S_l$ is dense. Hence, by compactness of $(\L_l,d_\L)$ (see Lemma \ref{CompactLoopFirst}), $\Cont_l$ extends to a $\mu[0,l]$-Lipschitz function on $[0,1]$.
%
%
%
\end{proof}
\subsection{Construction of the contour path $\Cont$.} \label{ConstructCont}
For every $f,g:[0,1]\mapsto \L$, let $d_\infty(f,g):=\max_{x\in[0,1]} d_\L(f(x),g(x))$.
In this section we prove:
\begin{proposition} \label{Cv cont} $(\Cont_{Y_a})_{a\in \N}$ is almost surely a Cauchy sequence for $d_{\infty}$.
\end{proposition} Since $\L$ is a.s. compact, this directly implies that $(\Cont_{Y_a})_{a\in \N}$ converges uniformly, and we define $\Cont$ as its limit. Our proof is mainly constructive, with several topological arguments along the way.
 
First, since almost surely as $a\to \infty$, $\mu[0,Y_a] Y_a/(\log^6 Y_a )\to \infty$, a.s. there exists for every $a\in \N$ large enough $(l_a,n_a)\in \R^+\times\N$ such that 
\begin{equation} 1 \leq \frac{Y_a}{\log^6 Y_a} \mu[0,Y_a] l_a\leq 2 \quad \text{and} \quad 2n_a l_a=1. \label{2/3/20ha} \end{equation}
we take $(l_a,n_a)$ such that $n_a$ is the largest possible. 

Then for every $a\in \N$, let $\x_{a,0}:=(0,0)$, $\x_{a,n_a}=(0,1)$ and for every $0<i<n_a$,  let $\x_{a,i}\in \L_{Y_a}$ such that 
 \begin{equation} \p_{Y_a,\Lft}(\x_{a,i})\in [2i l_a, (2i+1)l_a ]. \label{2/3/20hb}\end{equation}
$\x_{a,i}$ exists  by Lemma \ref{prelimLeft} (e), and by Lemma \ref{prelimLeft} (d), we can sample $\x_{a,i}$ in a measurable way. 
Also note that for every $a\in \N$, since $(0,0)$ is the minimum for $\prec$ and since $(0,1)$ is the maximum for $\prec$,
\[ \p_{Y_a,\Lft}(0,0)=0 \quad \text{and} \quad \p_{Y_a,\Lft}(0,1)=1. \] %
Thus, by Lemma \ref{prelimLeft} (c) and \eqref{2/3/20hb}, a.s. for every $a\in \N$ large enough, 
\begin{equation} (0,0)=\x_{a,0}\prec \x_{a,1}\dots \prec \x_{a,n_a}=(0,1). \label{3/3/18hc} \end{equation}
 
 Next, by Lemma \ref{Recall icrt} (a), (b), a.s. as $a\to \infty$, $n_a=O(Y_a^2)$. Also by Lemma \ref{Recall icrt} (c), a.s. $i^2=O(Y_i)$, so a.s. $\sum_{a=1}^\infty Y_a^{-3}<\infty$. Therefore by the Borel--Cantelli Lemma and Lemma \ref{E=mc2}, for every $a\leq b$ large enough, and $0\leq i <n_a$, by \eqref{2/3/20ha}, \eqref{2/3/20hb},
\begin{equation} l_a/2 \leq \p_{Y_b,\Lft}(\x_{a,i},\x_{a,i+1})\leq 4l_a. \label{2/3/20hc} \end{equation}
In particular for every $a\leq b$ large enough, $(\p_{Y_b,\Lft}(\x_{a,i}))_{0\leq i \leq n_a}$ is strictly increasing. Hence we may define $\psi_{a,b}:[0,1]\mapsto [0,1]$ such that for every $0\leq i\leq n_a$,
\begin{equation*} \psi_{a,b}(\p_{Y_a,\Lft}(\x_{a,i}))=\p_{Y_b,\Lft}(\x_{a,i}), \label{2/3/21h} \end{equation*}
and such that for every $0\leq i <n_a$, $\psi_{a,b}$ is linear in $[\x_{a,i},\x_{a,i+1}]$. And by \eqref{2/3/20hc}, $\psi_{a,b}$ is strictly increasing and continuous.
%
\begin{lemma} \label{Luigi=/=Boo} Almost surely for every $a\in \N$ large enough, $(\psi_{a,b})_{b\geq a}$ converges uniformly toward a strictly increasing continuous function $\psi_a$.  And a.s. $ (\psi^{-1}_{a,b} )_{b\geq a}$ converges uniformly toward $\psi^{-1}_a$.
\end{lemma}
\begin{proof}
By Lemma \ref{CVLeft} a.s. for every $a\in \N$ large enough, for every $0\leq i \leq n_a$, as $b\to \infty$, 
\begin{equation} \p_{Y_b,\Lft}(\x_{a,i})\to \p_{\Lft}(\x_{a,i}). \label{3/3/18h} \end{equation}
So if $\psi_a$ is the function such that for every $0\leq i\leq n_a$, 
 \begin{equation*} \psi_{a,b}(\p_{Y_a,\Lft}(\x_{a,i}))=\p_{\Lft}(\x_{a,i}), \label{2/3/21h} \end{equation*}
and such that for every $0\leq i <n_a$, $\psi_{a}$ is linear in $[\x_{a,i},\x_{a,i+1}]$. Then almost surely for every $a\in \N$ large enough, $(\psi_{a,b})_{b\geq a}$ converges uniformly toward $\psi_a$.

Moreover, by \eqref{2/3/20hc} and \eqref{3/3/18h}, a.s. for every $a\in \N$ large enough, for every $0\leq i <n_a$,
\begin{equation} l_a/2 \leq \psi_{a}(\p_{Y_a,\Lft}(\x_{a,i+1}))-\psi_{a}(\p_{Y_a,\Lft}(\x_{a,i}))\leq 4l_a. \label{3/3/18hb} \end{equation}
Hence, $\psi_a$ is also a strictly increasing continuous function.  Finally by \eqref{3/3/18h}, \eqref{3/3/18hb}, and the linearity, a.s. for every $a\in \N$ large enough, $ (\psi^{-1}_{a,b} )_{b\geq a}$ converges uniformly toward $\psi^{-1}_a$.
\end{proof}
%
%
%
%
\begin{lemma} \label{Odd job} Almost surely for every $a\leq b$ large enough,
\[ d_\infty ( \Cont_{Y_a}\circ \psi^{-1}_{a,b}, \Cont_{Y_b})\leq d_H(\L_{Y_a},\L)+6\log^6(Y_a)/Y_a. \]
\end{lemma}
\begin{remark} This result does not depends on the choice of $(\Cont_{Y_a})_{a\in \N}$, since they are unique up to $\sim_\L$.
\end{remark}
Before proving Lemma \ref{Odd job} let us explain why it implies Proposition \ref{Cv cont}. 
\begin{proof}[Proof of Proposition \ref{Cv cont}] First by Proposition \ref{FBLooptreeThm} a.s. $d_H(\L_{Y_a},\L)\to 0$. Hence, by Lemma \ref{Odd job} a.s.
\[ \lim_{a\to \infty} \max_{b\geq a} d_\infty ( \Cont_{Y_a}\circ \psi^{-1}_{a,b}, \Cont_{Y_b})=0. \]
Then by the triangular inequality, a.s. 
\[ \lim_{a\to \infty} \limsup_{b,c\to\infty} d_\infty (\Cont_{Y_c}, \Cont_{Y_b})\leq \lim_{a\to \infty} \limsup_{b,c\to\infty} d_\infty ( \Cont_{Y_a}\circ \psi^{-1}_{a,c}, \ \Cont_{Y_a}\circ \psi^{-1}_{a,b}). \]
Finally by Lemma \ref{Luigi=/=Boo} and continuity of $\Cont_{Y_a}$ the right hand side above is almost surely null. 
\end{proof}
%
\begin{proof}[Proof of Lemma \ref{Odd job}] Almost surely for every $a\leq b$ large enough, $\y\in \L_b$ the following holds: Since $\prec$ is an order, by \eqref{3/3/18hc} there exists $0\leq i<n_a$ such that $\x_{a,i}\prec \y\prec \x_{a,i+1}$. Then recall from Lemma \ref{BAKA NO TEST} that $\pL_{Y_a}(\y)$ denote the projection of $\y$ on $\L_{y_a}$. Then note that by Lemma \ref{KeepProjection}, $\x_{a,i}\prec \pL_{Y_a(\y)}\prec \x_{a,i+1}$. Thus,
 by Lemma \ref{prelimLeft} (c),
\begin{equation} \p_{Y_a,\Lft}(\x_{a,i}) \leq \p_{Y_a,\Lft}\circ \pL_{Y_a}(\y)  \leq \p_{Y_a,\Lft}(\x_{a,i+1}). \label{3/3/21h} \end{equation}

Also, since $\x_{a,i}\prec \y\prec \x_{a,i+1}$, 
\begin{equation*} \p_{Y_b,\Lft}(\x_{a,i}) \leq \p_{Y_b,\Lft}(\y)  \leq \p_{Y_b,\Lft}(\x_{a,i+1}). \end{equation*}
Then by applying $\psi^{-1}_{a,b}$ which is strictly increasing we get,
\begin{equation} \p_{Y_a,\Lft}(\x_{a,i}) \leq \psi^{-1}_{a,b}\circ \p_{Y_b,\Lft}(\y)  \leq \p_{Y_a,\Lft}(\x_{a,i+1}). \label{3/3/23h} \end{equation}

Therefore by \eqref{3/3/21h}, \eqref{3/3/23h}, \eqref{2/3/20hb}, and \eqref{2/3/20ha},
\[ |\psi^{-1}_{a,b}\circ \p_{Y_b,\Lft}(\y)-\p_{Y_a,\Lft}\circ \pL_{Y_a}(\y)|\leq 3 l_a\leq 6 \log^6 (Y_a)/(Y_a\mu[0,Y_a]).\]
Then since by Lemma \ref{First Cont}, $\Cont_{Y_a}$ is a.s. $\mu[0,Y_a]$-Lipschitz,
\[ d_\L(\Cont_{Y_a}\circ \psi^{-1}_{a,b}\circ \p_{Y_b,\Lft}(\y),\Cont_{Y_a}\circ \p_{Y_a,\Lft}\circ \pL_{Y_a}(\y)) \leq 6 \log^6 (Y_a)/Y_a.  \] 
Also, by using the definition of $\Cont_a$, we have, $\Cont_a\circ \p_{Y_a,\Lft}\circ \pL_{Y_a}(\y)\sim_\L \pL_{Y_a}(\y)$. Hence, 
\[ d_\L(\Cont_{Y_a}\circ \psi^{-1}_{a,b}\circ \p_{Y_b,\Lft}(\y),\pL_{Y_a}(\y)) \leq 6 \log^6 (Y_a)/Y_a.  \] 
Then using $d_\L(\pL_{Y_a}(\y),\y)\leq d_H(\L_{Y_a},\L)$ (see Lemma \ref{BAKA NO TEST}),
\[ d_\L(\Cont_{Y_a}\circ \psi^{-1}_{a,b}\circ \p_{Y_b,\Lft}(\y),\y) \leq d_H(\L_{Y_a},\L)+6 \log^6 (Y_a)/Y_a.  \] 
Finally by using the definition of $\Cont_a$, we have $\y\sim_\L \Cont_{Y_b}\circ \p_{Y_b,\Lft}(\y)$, so,
\begin{equation} d_\L(\Cont_{Y_a}\circ \psi^{-1}_{a,b}\circ \p_{Y_b,\Lft}(\y), \Cont_{Y_b}\circ \p_{Y_b,\Lft}(\y)) \leq d_H(\L_{Y_a},\L)+6 \log^6 (Y_a)/Y_a.  \label{4/3/0h}\end{equation}

To conclude the proof by \eqref{4/3/0h}, almost surely for every $a\leq b\in \N$ large enough, writing $S_b:=\{\p_{Y_b,\Lft}(\z),\z\in  \L_{Y_b}\}$,
\[ \max_{x\in S_b} d_\L(\Cont_{Y_a}\circ \psi^{-1}_{a,b}(x),\Cont_{Y_b}(x))\leq  d_H(\L_{Y_a},\L)+6 \log^6 (Y_a)/Y_a. \]
The maximum is then extended to $[0,1]$, by density of $S_b$ (see Lemma \ref{prelimLeft} (d)), since $\Cont_{Y_a}, \psi^{-1}_{a,b}, \Cont_{Y_b}$ are almost surely continuous.
\end{proof}
\subsection{Proof of Theorem \ref{ConstructContThm}}
Recall that by Lemmas \ref{ExtendPleft}, \ref{YOLONUL}, $\p_{\Lft}$ extends to a function continuous at each point of $\L\backslash (\R^+\times[0,1])$. We need to show that almost surely for every $\x\in \L$, $\Cont\circ \p_{\Lft}(\x)\sim_\L \x$. To this end, by continuity of $\p_{\Lft}$ and $\Cont$ at $\L\backslash (\R^+\times[0,1])$, and by density of $\R^+\times[0,1]$ (see Proposition \ref{Completion} (a)), it is enough to show the desired result on $\R^+\times[0,1]$.

Almost surely for every $\x\in \R^+\times[0,1]$ the following holds: First by definition of $(\Cont_l)_{l>0}$, for every $l$ large enough, $\Cont_l\circ \p_{l,\Lft}(\x)\sim_\L \x$. Then by Lemma \ref{CVLeft}, a.s. $\p_{l,\Lft}(\x)\to \p_{\Lft}(x)$. Hence, since a.s. $(\Cont_l)_{l>0}$ converges uniformly toward $\Cont$, as $l\to \infty$, 
\[ \Cont_l\circ \p_{l,\Lft}(\x)\to \Cont\circ \p_\Lft(\x)) \]
 Therefore, $\Cont\circ \p_\Lft(\x)\sim_\L \x$.
\subsection{Holder continuity of $\Cont$: Proof of Theorem \ref{HolderContThm}}
The next result is more precise than Theorem \ref{HolderContThm}, and we will use it in the next section to estimate the Minkowski lower box dimension of $\L$. 
\begin{lemma} Almost surely, for every $n\in \N$ large enough, for every $s,t\in [0,1]$,
\[ d_\L(\Cont(s),\Cont(t))\leq 13|s-t|\E[\mu[0,2^n]]+13n^62^{-n}. \]
\end{lemma}
\begin{proof}
We keep the notations of Section \ref{ConstructCont}. By taking $b\to \infty$ in Lemma \ref{Odd job}, by Lemma \ref{Luigi=/=Boo}, a.s. for every $a$ large enough, 
\[ d_\infty ( \Cont_{Y_a}\circ \psi^{-1}_{a}, \Cont)\leq d_H(\L_{Y_a},\L)+6\log^6(Y_a)/Y_a. \]
Hence, by the triangular inequality, a.s. for every $a$ large enough, for every $s,t\in [0,1]$,
\[ d_\L(\Cont(s),\Cont(t))\leq d_\L(\Cont_{Y_a}\circ \psi^{-1}_a(s),\Cont_{Y_a}\circ \psi^{-1}_a(t))+2d_H(\L_{Y_a},\L)+12\log^6(Y_a)/Y_a.\]
Moreover, by Lemma \ref{First Cont}, $\Cont_{Y_a}$ is $\mu[0,Y_a]$-Lipschitz, and by \eqref{2/3/20hb}, \eqref{3/3/18hb}, $\psi^{-1}_{a}$ is $6$-Lipschitz. Thus,
\begin{equation} d_\L(\Cont(s),\Cont(t))\leq 6|s-t|\mu[0,Y_a]+2d_H(\L_{Y_a},\L)+12l_a\mu[0,Y_a]. \label{4/3/2h} \end{equation}

Next, as a corollary of \cite{ICRT1} Lemma 4.5, a.s. for every $n$ large enough, there exists $a\in \N$, such that $Y_a\in [2^n,2^{n+1}]$. So, a.s. for every $n\in \N$ large enough, by \eqref{4/3/2h}, for every $s,t\in [0,1]$.
\begin{equation*} d_\L(\Cont(s),\Cont(t))\leq 6|s-t|\mu[0,2^{n+1}]+2d_H(\L_{2^n},\L)+12n^6 2^{-n}. \end{equation*}
The desired result follows from Lemma \ref{ChainingStart} (b), and by Lemma \ref{Recall icrt} (a) (b).
\end{proof}
\begin{proof}[Proof of Theorem \ref{HolderContThm}] Let $\delta>\badam$. By  \eqref{22/02/9h}, $\E[\mu[0,2^{n}]]=o(2^{n(\beta-1)})$. 
So, for every $n\in \N$ large enough, for every $s,t\in [0,1]$ with $2^{-n\beta} \leq |s-t|\leq 13 n^6 2^{-n\beta}$,
\begin{equation*} d_\L(\Cont(s),\Cont(t))\leq n^72^{-n}\leq - \log(|s-t|)^8|s-t|^{1/\beta}. \end{equation*}
Since $\beta>\badam$ is arbitrary, the desired result follows.
\end{proof}
\section{Fractal dimensions of the looptree} \label{DimensionsSec} 
\subsection{Definitions of the fractal dimensions} \label{RecallDefDimSec}
In this section $X$ denotes a pseudo-metric space. 
\begin{definition*} (Minkowski dimensions) For every $\e>0$, an $\e$-set of $X$ is a finite subset $S$ of $X$ such that $d_H(S,X)\leq \e$. For every $\e>0$ let $N_\e$ be the smallest size of a $\e$-set of $X$. Define the Minkowski lower box and upper box dimensions respectively by
\[ \underline{\dim}(X):=\liminf_{\e\to 0} \frac{ \log N_{\e}}{-\log \e} \quad \text{and} \quad \overline{\dim}(X):= \limsup_{\e\to 0} \frac{ \log N_{\e}}{-\log \e}. \]
\end{definition*}
\begin{definition*} (Packing dimension) For every $s\geq 0$ and $A\subset X$ let 
\[ P^s_0(A):= \limsup_{\delta\to 0} \left \{ \sum_{i\in I} \diam(B_i)^s \Bigg \vert \, \{B_i\}_{i\in I} \text{ are disjoint balls $B(x,r)$ with $x\in A$ and $0<r\leq \delta$}\right \}. \]
and
\[ P^s(X):=\inf \left \{\sum_{i=1}^{\infty} P^s_0(A_i) \Bigg \vert X\subset \bigcup_{i=1}^{\infty} A_i \right \}. \]
Then $P^s$ is a decreasing function of $s$, and we define the packing dimension of $X$ as 
\[ \dim_P(X):= \sup \{s, P^s(X)<\infty\}. \]
\end{definition*}
\begin{definition*} (Hausdorff dimension) For every $s,r\geq 0$ write
\[ H^s_r(X):= \inf_{\diam(A_i)\leq r} \left \{ \sum_{i=1}^{\infty} \diam(A_i)^s \Bigg \vert X \subseteq \bigcup_{i=1}^{\infty} A_i \right \}. \]
The Hausdorff dimension of $X$ is defined by
\[ \dim_H(X):=\sup \left \{s, \sup_{r\in \R^+}H_r^s(X)<\infty \right \}. \]
\end{definition*}
\begin{remark} Although, the above dimensions are usually considered for metric spaces, it is easy to check that they are exactly the same for a pseudo-metric space and for its quotient. For this reason, the below results still apply here.
\end{remark}
To compute the packing dimension and Hausdorff dimension we will use the following extension of Theorem 6.9, and Theorem 6.11 from \cite{fractal}. (\cite{fractal} deals with subsets of Euclidian space, but the same arguments hold for every pseudo-metric space.) 
\begin{lemma}\label{Hausdorff} Let $p$ be a Borel probability measure on $X$ and $s\in \R^+$.
 \begin{compactitem} 
  \item[a)]If $p$-almost everywhere $\limsup (\log p(B(x,\e)))/(\log \e)\geq s$ as $\e\to 0$, then $\dim_P(X)\geq s$.
 \item[b)]If $p$-almost everywhere $\liminf (\log p(B(x,\e)))/(\log \e)\geq s$ as $\e\to 0$, then $\dim_H(X)\geq s$.
 \end{compactitem}
\end{lemma}

We have well-known inequalities (see e.g. Chapter 3 of Falconer \cite{FalconPunch}):
\begin{lemma} \label{FalconPunch} For every pseudo-metric space $X$ we have
\[ \dim_H(X)\leq  \underline{\dim}(X) \leq \overline{\dim}(X) \quad \text{and} \quad \dim_H(X)\leq  \dim_P(X) \leq \overline{\dim}(X). \]
\end{lemma}
So we only need to upper bound $\underline{\dim}(\L)$, $\overline{\dim}(\L)$, and to lower bound $\dim_{H}(\L)$ and $\dim_{P}(\L)$.
\subsection{Upper bound on the Minkowski dimensions} \label{MinkowskiDimSec}
To upper bound the Minkowski dimensions we use the contour path $\Cont$. By Theorem \ref{ConstructContThm} for every $\alpha\in \L$, $\Cont(\p_\Lft(\alpha))\sim_\L \alpha$. So $d_H(\Cont([0,1],\L))=0$. As a result,  for every $x>0$, $\{ix,0\leq i \leq 1/x\}$ is a $\e_x$-set of $\L$, where $\e_x:=\max_{s,t}|\Cont(s)-\Cont(t)|$. So for every $\alpha>0$, if $\Cont$ is $\alpha$ H\"older continuous,
\[ \overline\dim(\L)\leq \limsup_{x\to 0} \frac{\log(1/x+O(1))}{-\log(\e_x)}\leq \limsup_{x\to 0} \frac{\log(1/x+O(1))}{(1+o(1))\log(1/x^\alpha)}=1/\alpha.\]
Therefore by Theorem \ref{HolderContThm} a.s. $\overline\dim(\L)\leq \badam$. Also,
\begin{lemma} \label{Minkowski<Cont} Almost surely for every $\alpha>1/\overline \dim(\L)$, $\Cont$ is not $\alpha$ H\"older continuous.
\end{lemma}

We now consider $\underline \dim(\L)$. By \eqref{22/02/9h}, there exists $(n_i)_{i\in \N}$ such that \smash{$\E[\mu[0,2^{n_i}]]=o(2^{n_i(\badim-1+o(1))})$}. So by Lemma \ref{HolderContThm}, a.s. for every $i$ large enough, $s,t\in [0,1]$ with $|s-t|\leq 2^{-n_i}/\E[\mu[0,2^{n_i}]]$, 
\[ d_\L(\Cont(s),\Cont(t))\leq 13|s-t|\E[\mu[0,2^{n_i+1}]]+13n_i^62^{-n_i}\leq n_i^72^{-n_i}.\]
Therefore, using the same $\e$-sets as before, a.s. as $i\to \infty$, 
\[ \underline\dim(\L)\leq \liminf_{i\to \infty} \frac{\log(2^{n_i}\E[\mu[0,2^{n_i}]])}{-\log(n_i^{7}2^{-n_i})}\leq \frac{n_i(1+\badim-1+o(1))}{n_i(1+o(1))}\leq  \badim.\]
\subsection{The rebranching principle.} \label{RerootSec}
We want to lower bound the dimensions of $\L$ with Lemma \ref{Hausdorff}. To this end, we morally needs to lower bound the distance between two random vertices in $\L$. 
In this section we show that it morally suffices to lower bound $d_\L((0,0),(Y_1,0))$.  
Our starting point is the rebranching principle, which we use here as follows: Let by convention $Y_0=0$.
\begin{proposition} \label{Permutation}For every permutation $\sigma$ of $\{0\}\cup \N$, we have the following joint equality in distribution,
\[ (d_\T(Y_i,Y_j))_{i,j\geq 0}\equal^{(d)} (d_\T(Y_{\sigma(i)},Y_{\sigma(j)}))_{i,j\geq 0},\] 
\[ \left (\1_{X_k\in \llbracket Y_i,Y_j\rrbracket }\right )_{i,j\geq 0, k\in \N}\equal^{(d)}  \left (\1_{X_k\in \llbracket Y_{\sigma(i)},Y_{\sigma(j)}\rrbracket }\right )_{i,j\geq 0,k\in \N}. \]
\end{proposition}
\begin{proof} We briefly recall some discrete notions of \cite{Uniform}. Let $(V_i)_{i\in \N}$ be a set of vertices. We say that $\D=(d_1,\dots, d_n)$ is a (pure) degree sequence if $\sum_{i=1}^n d_i=2n-2$ and if $d_1\geq d_2\geq \dots\geq d_n$. 

 For every degree sequence $\D=(d_1,\dots, d_n)$, we say that a tree $T$ have degree sequence $\D$ if $T$ has vertices $(V_i)_{1\leq i \leq n}$ and for every $1\leq i \leq n$, $V_i$ has degree $d_i$. For every degree sequence $\D$ let $T^\D$ denote a uniform tree with degree sequence $\D$. Also let $L_0^\D$, $L_1^\D$,\dots be the leaves of $T^\D$. Those leaves are determinist since they are the vertices of degree 1. We root $T^\D$ at $L_0^\D$.
 
For every tree $T$, let $d_T$ denote the graph distance in $T$.  Also, for every $A,B,C\in T$, we write $(A,T)\in \langle B,C\rangle$ if $A$ lies in the path between $B$ and $C$ in $T$. By Proposition 31 (b) in \cite{Uniform} Section 8.2, and by Lemma 26 (e) in \cite{Uniform} Section 8.1, there exists $(\D_n)_{n\in \N}$ some degree sequences, such that the number of leaves and vertices of degree at least $2$ diverges, and such that for every $a,b\in \N$ the following weak convergences hold jointly 
\begin{equation}  (d_{T^{\D_n}}(L^{\D_n}_i,L^{\D_n}_j) )_{0\leq i,j\leq a}\limit  (d_\T(Y_i,Y_j))_{0\leq i,j\leq a}, \label{19/03/15h} \end{equation}
\begin{equation} \left (\1_{(V_k,T^{\D_n})\in \langle L^{\D_n}_i,L^{\D_n}_j\rangle} \right )_{0\leq i,j\leq a, \, k\leq b}\limit  \left (\1_{X_k\in \llbracket Y_i,Y_j\rrbracket }\right )_{0\leq i,j\leq a, \, k\leq b}. \label{18/03/18h} \end{equation}

The main principle behind the rebranching principle is that for every $n$ large enough the laws of the left hand sides of \eqref{19/03/15h} and \eqref{18/03/18h} are invariant under the permutation of the leaves. So by limit the right hand sides must also be invariant under those permutations. 
\end{proof}
Our goal is now to deduce from Proposition \ref{Permutation} an identity for $d_\L$. First let us introduce some topological notions.
Recall that a.s. $\p$ is a probability Borel measure on $\T$. Let $\p_\L:=\p\times\1_{x\in [0,1]}dx$. 
 Let $\B_\L$ be the product topology on $\T\times[0,1]$. Let $\mathcal F:=\sigma(\mathbf X, \mathbf Y, \mathbf Z, (U_{X,i})_{i\in \N},(U_{Z,i})_{i\in \N})$. Let $\B_\R$ be the Borel topology on $\R$. Finally, for every $\x\in \T\times[0,1]$, $\e>0$, let $B(\x,\e)$ denote the open ball of center $\x$ of radius $\e$ for $d_\L$ in $\T\times[0,1]$.
 
 Keep in mind that we are using two levels of randomness. On the one hand, we work on a completed probability space $(\Omega, \mathcal F,\proba)$. On the other hand, we work at $\omega\in \Omega$ fixed, on the random probability space $(\T\times[0,1], \B_\L,\p_\L)$. 
\begin{lemma} \label{MeasurabilityA} Almost surely the following assertions hold:
\begin{compactitem}
\item[(a)] The map $\x,\y\mapsto d_\L(\x,\y)$ is $(\B_\L\times\B_\L,\B_\R)$-measurable.
\item[(b)] For every $(x,u)\in \T\times[0,1]$, $\e>0$, the open ball $B((x,u),\e)$ is $\B_\L$-measurable.
\item[(c)] For every $\e>0$, the map $\x\in \T\times[0,1]\mapsto \p_\L(B(\x,\e))$ is $(\B_\L,\B_\R)$-measurable.
\end{compactitem}
\end{lemma}
\begin{proof} Toward (a), for every $i\in \N$, the map $y\in \T \mapsto U_{X_i,y}$ is continuous on $\T\backslash \{X_i\}$, since it is locally constant. Then $v\mapsto U_{X_i,X_i,v}$ is continuous. Hence $(y,v)\in  \T\times[0,1]  \mapsto U_{X_i,y,v}$ is $(\B_\L,\B_\R)$-measurable. By definition of $d_\L$, (a) follows by sum and composition. 

Then by (a) the map $(x,u),(y,v)\mapsto \1_{d_\L((x,u),(y,v))\leq \e}$ is also $(\B_\L\times\B_\L, \B_\R)$-measurable. (b),(c) follows by Fubini's Theorem. 
\end{proof}
Next let $\mathcal M$ be the set of probability distribution on $\R^+$. We equip $\mathcal M$ with the weak topology. Recall that $\mathcal M$ is a Polish space. Let $\B_\M$ be the Borel topology on $\mathcal M$. For every probability distribution $\nu$ on $\T\times[0,1]$ $\B_\L$-measurable we construct $d_\L\star \nu\in \mathcal M$ as follows: Let $\x$, $\y$ be two random variables with law $\nu$. By Lemma \ref{MeasurabilityA} (a), $d_\L(\x,\y)$ is a random variable $(\B_\L\times\B_\L, \B_\R)$-measurable. Let $d_\L\star \nu\in \mathcal M$ denote its probability distribution.

\begin{lemma} \label{MeasurabilityB} Let $\x$ be a random variable with law $\p_\L$. 
We have:
\begin{compactitem}
\item[a)] For every $i,j\geq 0$, $d_\L((Y_i,0),(Y_j,0))$ is a random variable $(\mathcal F,\B_\R)$-measurable.
\item[b)] Almost surely,
\[ d_\L\star \left ( \frac{1}{n} \sum_{i=1}^n \delta_{Y_i,0} \right ) \limit_{\text{weakly}} d_\L\star \p_\L.\]
\item[c)] $d_\L\star \p_\L$ is a random variable $(\mathcal F,\B_\M)$-measurable.
\end{compactitem}
\end{lemma}
\begin{proof} Toward (a), let $i,j\in \N$. Note that $d_\T(Y_i,Y_j)$ is $(\mathcal F,\B_\R)$ measurable since by Algorithm \ref{Alg1} \linebreak[1] it is a sum of measurable random variables. Also by Algorithm \ref{Alg1}, for every $k\in \N$, $\1_{X_k\in \llbracket 0, Y_i\rrbracket }$ is $(\mathcal F,\B_\R)$-measurable. Then by the construction of the uniform angle function $U$ in Algorithm \ref{ConstructU}, either $X_j\notin \llbracket 0, Y_i\llbracket$ so $U_{X_k,Y_i,0}=0$, or $X_j\in \llbracket 0, Y_i\llbracket$ and then $U_{X_k,Y_i,0}=U_{X_k}$. Hence $U_{X_k,Y_i,0}$ is $(\mathcal F,\B_\R)$-measurable. Similarly,  $U_{X_k,Y_j,0}$ is $(\mathcal F,\B_\R)$-measurable. (a) follows by definition of $d_\L$.

Toward (b), recall that by \cite{ICRT1} Theorem 3.1, a.s. $\frac{1}{n}\sum_{i=1}^n \delta_{Y_i}$ converges weakly toward $\p$ for the weak topology on $\T$. Thus, a.s. if $(V_i)_{i\in \N}$ is a family of independent uniform random variables in $[0,1]$, $\frac{1}{n}\sum_{i=1}^n \delta_{Y_i,V_i}$ converges weakly toward $\p_\L$. Hence, a.s. 
\[ \frac{1}{n^2}\sum_{i=1}^n \sum_{i=1}^n \delta_{Y_i,V_i}\times \delta_{Y_j,V_j} \limit_{\text{weakly}}\p_\L\times \p_\L. \]
It directly follows by Lemma \ref{MeasurabilityA} that a.s. 
\begin{equation} d_\L\star \left ( \frac{1}{n} \sum_{i=1}^n \delta_{Y_i,V_i} \right ) \limit_{\text{weakly}} d_\L\star \p_\L.\label{19/03/14h} \end{equation}
Finally note that almost surely $\{Y_i\}_{i\in \N}\cap \{X_i\}_{i\in \N}=\emptyset$. So almost surely for every $i,j\in \N$, 
\[ d_\L((Y_i,0),(Y_j,0))=d_\L((Y_i,V_i),(Y_j,V_j)). \]
(b) follows from \eqref{19/03/14h}.

Toward (c), by (a) the left hand side of (b) is $(\mathcal F, \B_\M)$ measurable. So (c) follows from (b).
\end{proof}
Next, for every $f:\R^+\mapsto \R$ bounded continuous and $M\in \mathcal \M$ let $f(M):=\int_{\R^+} f(x)dM(x)$. Recall that $M\mapsto f(M)$ is continuous. So, by Lemma \ref{MeasurabilityB} (c), $f(d_\L\star \p_\L)$ is $(\F,\B_\R)$ measurable. 
\begin{proposition} \label{Rerooting} For every bounded continuous function $f:\R^+\mapsto \R$, 
 \[ \E[f(d_\L\star \p_\L)]=\E[f(d_\L((0,0),(Y_1,0)))]. \] 
\end{proposition}
\begin{proof}By Lemma \ref{MeasurabilityB} (b) and by Portmanteau's Theorem, it suffices to prove that for every $i\neq j$, $d_\L((Y_i,0),(Y_j,0))$ has the same law as $d_\L((0,0),(Y_1,0))$. Fix $i\neq j$. Since almost surely $\{Y_i\}_{i\in \N}\cap \{X_i\}_{i\in \N}=\emptyset$, by definition of $d_\L$, writing $\dc$ for the distance on the torus $[0,1]$,
\begin{equation} d_\L((Y_i,0),(Y_j,0))=d_\T(Y_i,Y_j)+\sum_{k\in \N}\theta_k \dc(U_{X_k,Y_i},U_{X_k,Y_j}). \label{18/03/5h} \end{equation}

Then for every $k\in \N$ with $X_k\notin \llbracket Y_i,Y_j\rrbracket$, $Y_i$ and $Y_j$ are connected in $\T\backslash \{X_k\}$ so $U_{X_k,Y_i}=U_{X_k,Y_j}$. Moreover, by the construction of the uniform angle function $U$ in Algorithm \ref{ConstructU}, note that conditionally on $(\mathbf X, \mathbf Y, \mathbf Z)$, $(\dc(U_{X_k,Y_i},U_{X_k,Y_j}))_{k:X_k\in \llbracket Y_i,Y_j\rrbracket}$ is a family of uniform random variables in $[0,1/2]$. 

As a result, by \eqref{18/03/5h}, conditionally on $(\mathbf X, \mathbf Y, \mathbf Z)$, $ d_\L((Y_i,0),(Y_j,0))-d_\T(Y_i,Y_j)$ is a sum of uniform random variables in $([0,\theta_k/2])_{k:X_k\in  \llbracket Y_i,Y_j\rrbracket}$. And similarly for $d_\L((0,0),(Y_1,0))-d_\T(0,Y_1)$. The desired result follows from Proposition \ref{Permutation}.
\end{proof}
\subsection{Lower bound on the Hausdorff and Packing dimensions}
To simplify the notations, let us write  $d_\L(0,Y_1)$ for $d_\L((0,0),(Y_1,0))$. Let 
\begin{equation} \badim':=\liminf_{\e\to 0} \frac{\log \proba(d_\L(0,Y_1)< \e)}{\log \e} \quad ; \quad \badam':=\limsup_{\e\to 0} \frac{\log \proba(d_\L(0,Y_1)<\e)}{\log \e}.\label{19/03/9h} \end{equation} 
\begin{remark} Keep in mind that many inequalities get reversed since $\log$ is negative on $(0,1)$.
\end{remark}
\begin{lemma} \label{LowerBoundHausdorff}  Almost surely $\dim_H(\L)\geq \badim'$ and $\dim_P(\L)\geq \badam'$.
\end{lemma}
\begin{proof} First note that $\T\times[0,1]\subset \L$, so 
\[ \dim_H(\T\times[0,1],d_\L)\leq \dim_H(\L) \quad \text{and} \quad \dim_P(\T\times[0,1],d_\L)\leq \dim_P(\L).  \]
Then recall Section \ref{RerootSec}. For every $n\in \N$, let $f_n:x\mapsto \max(\min(1,2-2^n x),0)$. Let $\x$ be a uniform random variable with law $\p_\L$. 
We have by Proposition \ref{Rerooting}, for every $n\in \N$,
\[ \E[f_n(d_\L\star p_\L)]=\E[f(d_\L(0,Y_1))]\leq \proba (d_\L(0,Y_1)<2^{-n+1}). \] 
Then by the Borel--Cantelli Lemma and Markov's inequality, a.s. for every $n$ large enough,
\[ f_n(d_\L\star p_\L)\leq n^2\proba(d_\L(0,Y_1)<2^{-n+1}). \] 
Also, note that a.s. if $\x$ is a random variable with law $p_\L$, so writing $\E_\L$ for the expectation with respect to $p_\L$, $\E_\L(B(\x,2^{-n}))\leq  f_n(d_\L\star p_\L)$. Hence, by the Borel--Cantelli Lemma and Markov's inequality, a.s. $\p_\L$ a.s.  for every $n$ large enough,
\[ B(\x,2^{-n})\leq n^4\proba(d_\L(0,Y_1)<2^{-n+1}). \] 
The desired result follows from Lemma \ref{Hausdorff}.
\end{proof}
\begin{lemma} Recall \eqref{22/02/9h}, \eqref{19/03/9h}. We have $\badim'\geq \badim$ and $\badam'\geq \badam$.
\end{lemma}
\begin{remark} Since a.s. $\underline{\dim}(\L)\leq \badim$ and  $\overline{\dim}(\L)\leq \badam$, by Lemmas \ref{FalconPunch}, \ref{LowerBoundHausdorff}, we have $\badim'\leq \badim$ and $\badam'\leq \badam$.
\end{remark}
\begin{proof} In order to simplify the main proof, let us first deal with the case $\theta_0>0$. By Lemma \ref{Recall icrt} (b), $\E[\mu[0,l]]\sim \theta_0^2l$ as $l\to \infty$. So $\badim=\badam=2$. 
Also, $d_\L(0,Y_1)\geq \theta_0^2/4 Y_1$ so for every $\e>0$,
\[ \proba(d_\L(0,Y_1)\leq \e)\leq \proba(Y_1\leq 4\e/\theta_0^2). \]
Then, since $\{Y_i\}_{i\in \N}$ is a Poisson point process of rate $\mu[0,x]dx$, as $\e\to 0$, by Lemma \ref{Recall icrt} (a),
\[ \proba(Y_1\leq \e)=\E[\proba(Y_1\leq \e|\mu)]\leq \e\E[\mu[0,\e]]=O(\e^2). \]
As a result, $\badim'\geq 2=\badim$ and $\badam'\geq 2=\badam$. So we may assume henceforth that $\theta_0=0$.

Next, note that  a.s. $d_\L(0,Y_1)=\sum_{0\leq X_i \leq Y_1} \theta_i \dc(U_{X_i})$. Thus, writing for every $n\in \N$, $\Delta_n: x\mapsto \sum_{0\leq X_i \leq Y_1} \min(\theta_i\dc(U_{X_i}) ,2^{-n})$, we have $d_\L(0,Y_1)\leq \Delta_n(Y_1)$.  Then writing for every $n\in \N$, $\e_n:=16n^2/\E[\mu[0,2^n]]$, we have, since $x\mapsto \Delta_n(x)$ is increasing,
\begin{equation} \proba(d_\L(0,Y_1)\leq 2^{-n} )\leq \proba(\Delta_n(\e_n)\leq 2^{-n})+\proba(Y_1\leq \e_n,\Delta_n(Y_1)\leq 2^{-n}). \label{19/03/11h} \end{equation}
Our proof consists in estimating both terms of the right hand side above.
\begin{remark} Let us morally explain why we work with $\Delta_n$ and not $d_\L$. We want to estimate, for $\e>0$ small, the typical $d_\L$-distance between two random vertices $\e$-$d_\T$-close from each other. It appears, that the few vertices $\e$-$d_\T$-close from the vertices of high degrees $X_1,X_2,\dots $, tends to be $d_\L$-far from each other. 
As a result, the moments of $d_\L(0,Y_1)\1_{Y_1\leq \e}$ are highly biased toward the moments of the typical $d_\L$-distance between two vertices $\e$-$d_\T$-close from $X_1,X_2,\dots$. To avoid this bias, we use $(\Delta_n)_{n\in \N}$ to truncate the $d_\L$-distance between the vertices $d_\T$-close from $X_1,X_2,\dots$.
\end{remark}
Now, recall that $(X_i)_{i\in \N}$ are independent exponential random variables of parameter $(\theta_i)_{i\in \N}$. We get with elementary computations, for every $0\leq x\leq 1$, $n\in \N$,
\begin{align} \E[\Delta_n(x)] = \sum_{i=1}^\infty \proba(X_i\leq x)\E[\min(\theta_i\dc(U_{X_i}) ,2^{-n})] 
 \geq  \frac{1}{8}\sum_{i=1}^\infty \theta_i x \min(\theta_i,2^{-n}).  \label{18/03/8ha}
\end{align}
And, for every $n\in \N$,
\begin{equation} \E[\mu[0,2^n]]=\sum_{i=1}^\infty \proba(X_i\leq 2^n)\theta_i =\sum_{i=1}^\infty (1-e^{-\theta_i2^n})\theta_i \leq \sum_{i=1}^\infty \min(1,\theta_i 2^n)\theta_i. \label{18/03/8hb}
\end{equation}
Hence, by \eqref{18/03/8ha} and \eqref{18/03/8hb}, for every $n\in \N$, $0\leq x \leq 1$, $\E[\Delta_n(x)]\geq x\E[\mu[0,2^n]]2^{-n-3}$. In particular, for every $n\in \N$ such that $\e_n\leq 1$, $\E[\Delta_n(\e_n)]\geq 2n^22^{-n}$. Also, note that $\Delta_n$ is a sum of independent random variables bounded by $2^{-n}$. So, by Bernstein's inequality (see e.g. \cite{Massart} Section 2.7 (2.10)) for every $n\in \N$ with $\e_n\leq 1$,
$\proba(\Delta_n(\e_n)\leq 2^{-n})\leq e^{-n^2/6}$.
Furthermore, by Lemma \ref{Recall icrt} (c), $\E[\mu[0,2^n]]=O(2^n)$. So, as $n\to \infty$ with $\e_n\leq 1$,
\begin{equation} \proba(\Delta_n(\e_n)\leq 2^{-n})=O(n^22^{-n}/\E[\mu[0,2^n]). \label{19/03/12h}\end{equation}

Toward upper bounding the right most term of \eqref{19/03/11h}, let us recall the construction of Aldous, Camarri, Pitman \cite{introICRT1, introICRT2} of the ICRT in the simple case where $\theta_0=0$. Let $((A_{i,j} )_{j\geq 0})_{i\in \N}$ be a family of independent Poisson point processes of intensity $(\theta_i)_{i\in \N}$ on $\R^+$. We proved in \cite{ICRT1} Lemma 2.1 that there exists a coupling such that a.s.
\[ \left ((X_i)_{i\in \N},\{Y_i,Z_i\}_{i\in \N} \right)\equal \left ( (A_{i,0})_{i\in \N}, \{A_{i,j},A_{i,0} \}_{i,j\in \N} \right). \]
In particular $Y_1\in \{A_{i,1}\}_{i\in \N}$
So, by an union bound, for every $x\geq 0$,
\begin{align*} \proba(Y_1\leq x,\Delta_n(Y_1)\leq 1/2^n) & \leq \sum_{i\in \N} \proba(A_{i,1}\leq x, \Delta_n(Y_1)\leq 2^{-n}) \notag
\\ & \leq \sum_{i\in \N} \proba(A_{i,1}\leq x, \theta_i U_{X_i}\leq 2^{-n}) \notag
\\ & \leq \sum_{i\in \N} x^2 \theta_i^2 \min(1, 2^{-n}/\theta_i). \notag
\end{align*}
Also, by directly adapting \eqref{18/03/8hb}, we get for every $n\in \N$, $\E[\mu[0,2^n]]\geq (1/4) \sum_{i=1}^\infty \min(1,\theta_i 2^n)\theta_i$. So, for every $x\geq 0$,
\begin{equation} \proba(Y_1\leq x,\Delta_n(Y_1)\leq 1/2^n) \leq 4 x^2 \E[\mu[0,2^n]]2^{-n}. \label{19/03/12hb} \end{equation}
In particular, since $\e_n=16n^2/\E[\mu[0,2^n]]$,
\begin{equation}  \proba(Y_1\leq \e_n,\Delta_n(Y_1)\leq 1/2^n)\leq 4 \e_n^2 \E[\mu[0,2^n]]2^{-n} =O(n^42^{-n}/\E[\mu[0,2^n]]). \label{19/03/13h}\end{equation}

To sum up, by \eqref{19/03/11h}, \eqref{19/03/12h}, \eqref{19/03/13h}, we have as $n\to \infty$ with $\e_n\leq 1$,
\begin{align} \proba(d_\L((0,Y_1)\leq 2^{-n} )=O(n^42^{-n}/\E[\mu[0,2^n]]). \label{20/03/12ha} \end{align}

Finally,  we have as $n\to \infty$ with $\e_n\geq 1$, using, $Y_1\leq A_{1,1}$, \eqref{19/03/12hb}, and $\e_n=16n^2/\E[\mu[0,2^n]]$,
\begin{align} \proba(d_\L(0,Y_1)\leq 2^{-n} ) & \leq \proba(Y_1>n^2)+\proba (Y_1\leq n^2, \Delta_n(Y_1)\leq 2^{-n}) \notag
\\ & \leq \proba(A_{1,1}>n^2)+\proba (Y_1\leq n^2, \Delta_n(Y_1)\leq 2^{-n}) \notag
\\ & = O(n^2 e^{-\theta_1 n^2}+ n^4 \E[\mu[0,2^n]]2^{-n}  ) \notag
\\ & =O(n^82^{-n}/\E[\mu[0,2^n]]). \label{20/03/12hb}
\end{align} 
The desired result follows from \eqref{20/03/12ha} \eqref{20/03/12hb} and the definitions of $\badam,\badim$ in \eqref{22/02/9h}, and of $\badam',\badim'$ in \eqref{19/03/9h}.
\end{proof}
\section{Proof of Theorem \ref{CompactGFF}}  \label{CompactGFFSec}
%
Recall the definition of $\GFF$ from \eqref{22/12/00}. Note that $\GFF$ is continuous on any bounded interval of $\R^+$. We need to prove that it extends to a continuous map on $\T$. To this end, we use the chaining method. Let $\mathcal F=\sigma(\mu, \mathbf Y, \mathbf Z)$. For all $x\in \T$, $y\in \R^+$, recall that $\pT_y(x)$ denote the projection of $x$ on $[0,y]$. Also to simplify the notations, for every $x,y\in \R^+$ let $d_\GFF(x,y):=\max_{z\in \llbracket x , y\rrbracket}|\GFF(x)-\GFF(z)|$. 
\begin{remark} Although $d_\GFF$ is a (pseudo) distance on $\T$, we will not study the topology of $(\T,d_\GFF)$.
\end{remark}
\begin{lemma} \label{BAKA_GFF}  a) For every $x,y$ random variables in $\R^+$ $\mathcal \F$-measurable and $t\geq 0$, almost surely 
\[ \proba  \left ( \left . d_\GFF(x,y)>t\sqrt{d_\T(x,y)} \right |\mathcal F  \right )\leq 4e^{-t^2/2}. \] 
b) For every $x,y,t\geq 0$, 
\[ \proba \left (d_\GFF(\pT_y(x),x)>t^2/\sqrt{\mu[0, y]} \right)\leq 5e^{-t^2/4}. \]
\end{lemma}
\begin{proof} Toward (a). Let $\gamma:[0,d_\T(x,y)]\mapsto \llbracket x,y\rrbracket$ be the geodesic from $x$ to $y$.  Note that by definition of $\GFF$, conditionally on $\mathcal F$, $\GFF \circ \gamma-\GFF(x)$ is a brownian motion. So (a) directly follows from standard results on Brownian motions.

Toward (b), by \cite{ICRT1} Lemma 4.7 for every $x,y,t\geq 0$, 
\[ \proba\left (d_\T(\pT_y(x),x)>t^2/\mu[0,y] \right )\leq e^{-t^2/4}. \] 
 The desired result follows using (a) and an union bound.
\end{proof}
We may assume below that $\theta_0>0$ or $\sum_{i=1}^\infty \theta_i=\infty$, since otherwise $\E[\mu(\R^+)]<\infty$ and the assumption of Theorem \ref{CompactGFF} is not satisfied. 
Then by Lemma \ref{Recall icrt} (a), we may define for every $n\in \N$, $\X_n$ as the unique real such that $\E[\mu[0,\X_n]]=e^n$. The next result is adapted from \cite{ICRT1} Lemma 6.2, which is used to prove the compactness of ICRT and upper bound its fractal dimensions. 
\begin{lemma} \label{ChainingStartGFF} Almost surely for every $n$ large enough:
\[ \max_{x\in [0,\X_{n}]} |\GFF(x)-\GFF(\pT_{\X_{n-1}}(x))|\leq 21 \log(\X_n) e^{-(n-1)/2}.\]
\end{lemma}
\begin{proof} 
 We first adapt the proof of \cite{ICRT1} Lemma 6.2, to morally deal with vertices far from $\{Y_i\}_{i\in \N}$. For every $n\in \N$, and $x\in\R^+$ let $E_n(x)$ be the event $d_\GFF(\pT_{\X_{n-1}}(x),x)>20 \log(\X_n)e^{-(n-1)/2}$. By Fubini's Theorem,  Lemma \ref{BAKA_GFF} (b) and $\mu[0,\X_{n-1}]=e^{n-1}$, we have
\[ \E \left [ \int_{0}^{\X_{n}} \1_{E_n(x)} dx \right  ]  = \int_{0}^{\X_{n}} \proba \left (E_n(x) \right ) dx  \leq 5 \X_{n} \exp \left ( -5 \log \X_{n} \right )=5\X_n^{-4}.
\]
Furthermore by Lemma \ref{Recall icrt} (b), $e^n=O(\X_n)$ so $\sum \X_n^{-1}<\infty$. Hence by Markov's inequality and the Borel--Cantelli Lemma, for every $n$ large enough,
 \[  \int_{0}^{\X_{n}} \1_{E_n(x)}dx < \X_{n}^{-3}. \]
 Note that it implies that, for every $n$ large enough, for every $x\in [0,\X_{2^n}]$, $y\in [0,\X_{2^n}]$, with $x\in \llbracket p_{\X_{n-1}}(y),y \rrbracket$ and $d_\T(x,y)\geq \X_{n}^{-3}$,
\begin{equation} |\GFF(x)-\GFF(\pT_{\X_{n-1}}(x))|\leq 20 \log(\X_n)e^{-(n-1)/2}, \label{22/12/13h} \end{equation}
since otherwise for every $z\in \llbracket x,y\rrbracket$ we would have $E_n(z)$. 

Next, let $\Lambda_n:=(\{Y_i\}_{i\in \N}\cap [\X_{n-1},\X_n])\cup \{ \X_{n}\}$. Note that for every $x\in  [0,\X_n]$, if there exists $y\in \Lambda_n$ such that $x\in \llbracket \pT_{\X_{n-1}}(y),y\rrbracket$ and $d_\T(x,y)>\X_n^{-3}$, then \eqref{22/12/13h} holds. So, writing for $y\in \Lambda_n$, $B(y,\X_{n}^{-3})$ for the closed ball of center $y$ and radius $\e$, and $\gamma_{n,y}:=  \llbracket \pT_{\X_{n-1}}(y),y\rrbracket \cap B(y,\X_{n}^{-3})$,  it remains to estimate $\GFF$ on the set $\bigcup_{y\in \Lambda_n} \gamma_{n,y}$.
%
%
%
%
%

More precisely, by \eqref{22/12/13h}, it is enough to show that a.s. for every $n$ large enough and $y\in \Gamma_n$, 
\begin{equation} M_{n,y}:= \max_{x\in \gamma_{n,y}}|\GFF(x)-\GFF(y)|\leq \log(\X_n)e^{-(n-1)/2}. \label{23/02/7h} \end{equation}
To this end, note that for every $y\in \Gamma_n$, $\gamma_{n,y}$ is a geodesic with extremities $\F$-measurable, and  of length at most $\X_n^{-3}$. Thus, by Lemma \ref{BAKA_GFF} (a) and an union bound a.s., 
\begin{equation} \proba\left (\left . \max_{y\in \Gamma_n} M_{n,y}>\log(\X_n)e^{-(n-1)/2} \right |\mathcal F \right ) \leq \#\Gamma_n 4e^{-\X_n^{3}\log(\X_n)^2 e^{-(n-1)}}. \label{27/12/18h} \end{equation}
Then by Lemma \ref{Recall icrt} (b), $e^n=O(\X_{n})$, and by Lemma \ref{Recall icrt} (c), $\#\Gamma_n=O(\X_n^2)$. Thus a.s. the right hand side of \eqref{27/12/18h} is summable. Finally \eqref{23/02/7h} follows from the Borel--Cantelli Lemma.
\end{proof}
\begin{proof}[Proof of Theorem \ref{CompactGFF}] Beforehand, let us rewrite the assumption. We have,
\[ \int_{\X_1}^{\infty} \frac{dl}{l\sqrt{\E[\mu[0,l]]}} = \sum_{n=1}^{\infty} \int_{\X_{n}}^{\X_{n+1}} \frac{dl}{l\sqrt{\E[\mu[0,l]]}} \geq \sum_{n=1}^{\infty} \int_{\X_{n}}^{\X_{n+1}} \frac{dl}{le^{{(n+1)}/2}} 
\geq \sum_{n=2}^{\infty} \frac{\log \X_{n}}{e^{{(n+9)}/2}}-\log \X_1. \]
So the assumption implies that $\sum \log \X_{n} e^{-n/2}<\infty$. (It is actually equivalent.)

Then for every $n\in \N$ let $\GFF_n:x\in \T\mapsto \GFF(\pT_{\X_{n-1}}(x))$. Note that for every $n\in \N$, since $\GFF$ is continuous on $([0,2^n],d_\T)$, $\GFF_n$ is continuous. Moreover, by Lemma \ref{ChainingStartGFF}, a.s. for every $n$ large enough $\|\GFF_n-\GFF_{n-1}\|_{\infty}\leq 21\log \X_{n} e^{-(n-1)/2}$. Therefore, by the assumption of Theorem \ref{CompactGFF}, $\{\GFF_n\}_{n\in \N}$ is a.s. a Cauchy sequence of continuous function, and so converges uniformly toward a continuous function $\tilde \GFF$  on $\T$. Finally note that $\tilde \GFF$ extends $\GFF$.
\end{proof}
\section{Proof and extension of Theorem \ref{CompactFennec}} \label{CompactFennecSec}
We consider $(\Clown_i)_{\in \N}$ independent random functions from $[0,1]$ to $\R$ independent with $\sigma(\mathbf X,\mu,\mathbf Y,\mathbf Z, U)$. 
We assume that:
\begin{Hypo}\label{Locally centered} For every $i\in \N$, $\Clown_i(0)=0$ and $\E[ \Clown_i(U_{X_i})]=0$.
\end{Hypo} 
\begin{Hypo}\label{Maximum 4 moment} \label{4+e moment} There exists $\kappa\geq 2$ such that 
 for every $i\in \N$, $\E\left [\|\Clown_i\|_\infty^{\kappa} \right]\leq1$, and
  \begin{equation} \E[\mu[0,l]]=O(l^{\kappa/2-1}\log(l)^{-5\kappa}). \label{28/02/18h} \end{equation}
\end{Hypo}
\begin{remarks} The assumptions are not necessary, and our method may be extended to other functions. However, they naturally appear in the study of field on $\D$-trees: (we will give more details in \cite{Dsnake})
\\$\bullet$ The first assumption is called locally centered. One can always split each $\Clown_i$, into a "constant" $u\mapsto \1_{u\neq 0} c_i$, and a centered function. However, dealing with constants require different techniques. 
\\ $\bullet$ The second assumption is an improvement of the so called $4+\e$ moment assumption. Indeed by Lemma \ref{Recall icrt} (a), $\E[\mu[0,2^n]]=O(2^n)$, so when $\kappa>4$, \eqref{28/02/18h} holds. 
\\ $\bullet$ In particular if $\mu$ corresponds to the $\alpha$-stable trees, $\mu[0,l]\sim c.l^{\alpha-1}$, and we can consider $\kappa>2\alpha$. 
\end{remarks}
Under Assumptions \ref{Locally centered}, \ref{4+e moment}, which are satisfied by $(\Brown_i)_{i\in \N}$ (up to a constant), we have:
\begin{proposition} \label{UniformConvergence}Almost surely $\sum_{i=1}^\infty \sqrt{\theta_i} \Clown_i(U_{X_i,\x})$ converges uniformly on $\L$. 
\end{proposition} 
For every $n,m \in \N$ let 
\[ \Slow_{n,m}:= \sum_{i=n}^m \sqrt{\theta_i} \Clown_i(U_{X_i,\x}). \]
 To prove the uniform convergence, we need to show that almost surely as $N\to \infty$,
\[ \Delta_N:= \max_{\x\in \L}\max_{n,m\geq N} \left |\Slow_{n,m}(\alpha) \right |\to 0.\]
Note that for every $n,m\in \N$, $|\Slow_{n,m}(\alpha)|$ reach its maximum on $\bigcup_{i=n}^m \{X_i\}\times[0,1]\subset \R^+\times[0,1]$. So,
\begin{equation} \Delta_N=\max_{(x,u)\in \R^+\times[0,1]}\max_{n,m\geq N} |\Slow_{n,m}(x,u)|. \label{RemarkCVU} \end{equation}

To show that a.s. $\Delta_N\to 0$ we use the chaining method (see Section \ref{1.3}): For every $i\in \N$, let 
\[ \Delta_{N,i}:= \max_{Y_i< x\leq Y_{i+1},u\in [0,1]}\max_{n,m\geq N} |\Slow_{n,m}(x,u)-\Slow_{n,m}(Z_i, U_{Z_i,Y_{i+1}})|\]
Then define $\MAX_N$ by induction on $([0,Y_i])_{i\in \N}$ such that $\MAX_N(0)=0$, and such that for every $i\geq 0,x\in ]Y_{i},Y_{i+1}]$, 
\begin{equation} \MAX_N(x):=\MAX_N(Z_i)+\Delta_{N,i}. \label{24/02/14h} \end{equation}
 The triangular inequality implies by induction that for every $i\in \N$, for every $x\in[0,Y_{i}]$, $u\in[0,1]$, $\max_{n,m\geq N} |\Slow_{n,m}(x,u)|\leq \MAX_N(x)$, so
\begin{equation} \Delta_N \leq \max_{x\in \R^+}  \MAX_N(x). \label{24/02/19hd}\end{equation}
 Our method consists in estimating $(\Delta_{N,i})_{N,i\in \N}$, and then deduce from \eqref{24/02/14h} and Lemma \ref{CountingLemma} an estimate on $\max \MAX_N$.
 
 To simplify the notations we write $\mathcal F:=\sigma(\mathbf X,\mu,\mathbf Y,\mathbf Z)$ and $\mu^N:=\sum_{i\geq N}\theta_i \delta_{X_i}$.
\begin{lemma} \label{DEGUEUX} There exists $C_\kappa >0$ which depends only on $\kappa$ such that for every $N,i\in \N$, $t>0$,
\[ \proba\left (\left. \Delta_{N,i}\geq  t\right |\mathcal F \right )\leq C_\kappa\left (\frac{\mu^N(Y_i,Y_{i+1}]}{t^2}\right)^{\kappa/2}.\]
\end{lemma}
\begin{proof} We work conditionally on $\mathcal F$. First for every $x\in (Y_i,Y_{i+1}]$ and $j\in \N$ with $X_j\notin \{Z_i\}\cup (Y_i,x]$, $Z_i$ and $x$ are connected in $\T\backslash \{X_j\}$ so $U_{X_j,Z_i}=U_{X_j,x}$. Similarly, $U_{Z_i,x}=U_{Z_i,Y_{i+1}}$. Also, by definition of $U$, for every $j\in\N$ with $X_j\in(Y_i,x)$, $U_{X_j,x}=U_{X_j,Y_{i+1}}=U_{X_j}$.  Also, $U_{x,x,0}=0$. Therefore, for every $n,m\in \N$, and $x\in (Y_i,Y_{i+1}]$,
\[ \Slow_{n,m}(x,0)-\Slow_{n,m}(Z_i, U_{Z_i,Y_{i+1}})=\sum_{j\in \N:Y_i<X_j< x} \sqrt{\theta_j}\Clown_j(U_{X_j}). \]

Moreover by definition of $\Clown$, $U$ and Assumption \ref{Locally centered}, $(\Clown_j(U_{X_j}))_{j\in \N}$ are independent and centered. Also by Assumption \ref{4+e moment}, for every $j\in \N$, $\E[\|\Clown_j\|_\infty^{\kappa}]\leq 1$. So by Lemma \ref{4+e}, there exists $C>0$, which depends only on $\e$, such that for every $N\in \N$, and $t>0$, a.s.
\begin{equation} \proba\left (\left. \max_{n,m\geq N} \max_{Y_i<x\leq Y_{i+1}} \left |\Slow_{n,m}(x,0)-\Slow_{n,m}(Z_i, U_{Z_i,Y_{i+1}})\right |\geq  t \right | \mathcal F \right )\leq C\left(\frac{\mu^N(Y_i,Y_{i+1}]}{t^2} \right )^{\kappa/2}. \label{24/02/16h} \end{equation}

Furthermore, for every $x\in (Y_i,Y_{i+1}]$, $u\in[0,1]$, with $x\notin\{X_j\}_{j\in \N}$, for every $n,m\in \N$, $\Slow_{n,m}(x,u)=\Slow_{n,m}(x,0)$. And for every $j\in \N$,
\[ \max_{n,m\geq N} \max_{u\in [0,1]} \left |\Slow_{n,m}(X_j,u)-\Slow_{n,m}(X_j, 0) \right | = \1_{j\geq N} \sqrt{\theta_j} \max_{u\in[0,1]} \Clown_j(u).\]
Then by Assumption \ref{Maximum 4 moment}, and by Markov's inequality, for every $j\in \N$,  $t>0$,
\[ \proba \left (\sqrt{\theta_j} \max_{u\in[0,1]} \Clown_j(u)>t \right )\leq (\theta_j/t^2)^{\kappa/2}. \]
Therefore by an union bound, and by convexity of $x\mapsto x^{\kappa/2}$, for every $t>0$,
\begin{align} \proba \left (\left . \max_{n,m\geq N} \max_{Y_i<x\leq Y_{i+1}}\max_{u\in[0,1]}\left |\Slow_{n,m}(x,u)-\Slow_{n,m}(x, 0) \right |\geq t \right | \mathcal F \right ) & \leq \sum_{j\geq N} \left (\frac{\theta_j}{t^2} \right )^{\kappa/2}\1_{Y_i<X_j\leq Y_{i+1}} \notag
\\ & \leq \left(\frac{\mu^N(Y_i,Y_{i+1}]}{t^2} \right )^{\kappa/2} . \label{24/02/14hb} \end{align}

Finally the desired inequality follows from \eqref{24/02/16h}, \eqref{24/02/14hb}, and an union bound.
\end{proof}
\begin{lemma} \label{fatigue} A.s. for every $n$ large enough, for every $i\in \N$ with $(Y_i,Y_{i+1}]\cap [2^n,2^{n+1}]\neq \emptyset$, 
\[ \Delta_{1,i}\leq \lambda_n:= n^2 \mu[0,2^n]^{1/\kappa} 2^{n/\kappa-n/2}.\]
\end{lemma}
\begin{proof} A.s. for every $n$ large enough and $i\in \N$ with $(Y_i,Y_{i+1}]\cap [2^n,2^{n+1}]\neq \emptyset$, we have by Lemma \ref{Recall icrt} (d) 
\[ \mu(Y_i,Y_{i+1}]\leq \log(Y_{i+1})^2/Y_{i+1}\leq n^2/2^n. \]
 Then by Lemma \ref{Recall icrt} (c), then (a) and (b), a.s. for every $n$ large enough,
 \[ \#\{i\in \N, (Y_i,Y_{i+1}]\cap [2^n,2^{n+1}]\neq \emptyset\} \leq 2\mu[0,2^{n+1}]2^{n+1}\leq 9\mu[0,2^n]2^n.\]
 
Therefore, by Lemma \ref{DEGUEUX} a.s. for every $n$ large enough, writing $E_n$ for the event that there exists $i\in \N$ such that $(Y_i,Y_{i+1}]\cap [2^n,2^{n+1}]\neq \emptyset$, and such that $\Delta_{1,i}>\lambda_n$,
\[ \proba(E_n|\mathcal F)\leq 9C_\kappa \mu[0,2^n]2^n \left ( \frac{n^2/2^n}{n^4 \mu[0,2^n]^{2/\kappa} 2^{2n/\kappa-n}}\right )^{\kappa/2}=9C_\kappa n^{-\kappa}. \]
The desired result follows by the Borel--Cantelli Lemma.
\end{proof}
\begin{proof}[Proof of Proposition \ref{UniformConvergence}.] Recall that we need to prove that a.s. as $N\to \infty$, $\|\MAX_N\|_{\infty}\to 0$. 
First, by Lemma \ref{CountingLemma} a.s. for every $n\in \N$ large enough, every $x\in [0,2^{n+1}]$ is separated from $[0,2^n]$ by at most $4n$ branches. So, by Lemma \ref{fatigue} a.s. 
 for every $n$ large enough, for every $N\in \N$,
 \[ \max_{x\leq 2^{n+1}} \MAX_N(x) \leq \max_{x\leq 2^n} \MAX_N(x) +4n\lambda_n \]
 Hence, a.s. for every $n$ large enough, for every $N\in \N$,
\begin{equation} \max_{x\in \R^+} \MAX_N(x)\leq \max_{x\leq 2^n} \MAX_N(x)+\sum_{i=n}^\infty 4i\lambda_i. \label{24/02/19ha}\end{equation}

Then note that for every $i\in \N$, $(\Delta_{N,i})_{N\in \N}$ is decreasing. So by Lemma \ref{DEGUEUX}, since $\mu$ is a.s. locally finite, a.s. as $N\to \infty$,  $\Delta_{N,i}\to 0$.
Therefore a.s. for every $n\in \N$, as $N\to \infty$,
\begin{equation} \max_{x\leq 2^n} \MAX_N(x) \leq \sum_{i:Y_i\leq 2^n} \Delta_{N,i} \to 0. \label{24/02/19hb} \end{equation}

Moreover by Lemma \ref{Recall icrt} (a), and by Assumption \ref{4+e moment}, a.s. as $n\to \infty$,
\begin{equation} \sum_{i\geq n}4i\lambda_i  = \sum_{i\geq n}4i^3 \mu[0,2^i]^{1/\kappa} 2^{i/\kappa-i/2} \sim  \sum_{i\geq n}4i^3 \E[\mu[0,2^i]]^{1/\kappa} 2^{i/\kappa-i/2}\to 0. \label{24/02/19hc} \end{equation}

To conclude the proof, by \eqref{24/02/19ha},\eqref{24/02/19hb},\eqref{24/02/19hc} almost surely as $N\to \infty$, $\|\MAX_N \|_{\infty}\to 0$. 
\end{proof}
%

\bibliographystyle{unsrt}
\appendix
\appendix
\section{Holder continuity of a Gaussian free field.}

If $(X,d)$ is a metric space, we call a Gaussian free field on $X$, a random function $\F:X\mapsto \R^+$ such that for every $x\in X$, $\F(x)$ is measurable, and such that for every $x,y\in X$, $\F(x)-\F(y)$ is a Gaussian random variable with variance $d(x,y)$. 
\begin{lemma} \label{HolderXD} Let $(X,d)$ be a metric space. Let $\F$ be a Gaussian free field on $X$. If $\F$ is almost surely continuous and $(X,d)$ have finite upper Minkowski dimension then almost surely $\F$ is H\"older continuous with any exponent smaller than $1/2$.
\end{lemma}
\begin{proof} Since $(X,d)$ have finite upper Minkowksi dimension, there exists $k\in \N$ such that for every $n\in \N$, there exists $a_{n,1},\dots, a_{n,k^n}\in X$ such that $\max_{x\in X} d(x,\{a_{n,1},\dots, a_{n,k^n}\})\leq 2^{-n}$. 

For every $n\in \N$, let $E_n$ denote the following event:
\[ E_n:=\left \{\forall i\leq k^n, \forall j\leq k^{n+1}, d(a_{n,i},a_{n+1,j})\leq 2^{9-n}\ply |\F(a_{n,i})-\F(a_{n+1,j})|\leq k n 2^{-n/2} \right \}.\] 
By an upper bound, and by definition of a Gaussian free field, we have for every $n\in \N$,
\[ \proba(E_n)\leq k^{n}k^{n+1}e^{-(8kn2^{-n/2})^2/(2.2^{9-n})}=k^{2n+1}e^{-k^2 n^2/2^{18}}.\]
Thus, since the right hand side is summable, by the Borel--Cantelli Lemma, almost surely for every $n\in \N$ large enough we have $E_n$.

The rest of the proof is deterministic. Assume that there exists $N\in \N$ such that for every $n\geq N$ we have $E_N$, and that $\F$ is continuous. Let $x,y\in X$ such that $d(x,y)\leq 2^{-N-9}$. Let $n\in \N$ such that $2^{-n+1}<d(x,y)\leq 2^{-n}$. Since $\F$ is continuous we can consider $m\in \N, m>n$ such that for every $z\in X$ with $d(x,z)\leq 2^{-m}$, $|\F(x)-\F(z)|\leq 2^{-n/2}$, and similarly for $y$. Let by induction $x_{m+1}, x_{m},x_{m-1}, \dots, x_n$ such that $x_{m+1}=x$, and such that for every $m\geq i\geq n$, $x_{i}\in \{a_{i,1},\dots, a_{i,k^i}\}$, and $d(x_i,x_{i+1})\leq 2^{-i}$. Define similarly $y_{m+1},\dots, y_n$. We have 
\[ d(x_n,y_{n+1})\leq \sum_{i=n}^m d(x_{i+1},x_{i})+d(x,y)+\sum_{i=n+1}^{m}d(y_i,y_{i+1})\leq \sum_{i=n}^m 2^{-i}+2^{-n}+\sum_{i=n+1}^{m}2^{-i}\leq 2^{2-n}. \]
So by the triangular inequality, then by definition of $m$ and by $E_n, E_{n+1},\dots, E_m$,
\begin{align*} |\F(x)-\F(y)|  \leq & |\F(x_{m+1})-\F(x_{m})|+\sum_{i=n}^{m-1} |\F(x_{i+1})-\F(x_{i})|
\\ & +|\F(x_n)-\F(y_{n+1})|+\sum_{i=n+1}^{m-1} |\F(y_i)-\F(y_{i+1})|+|\F(y_{m})-\F(y_{m+1})| 
\\   \leq & 2^{-n/2}+\sum_{i=n}^{m-1}k i 2^{-i/2}+k n 2^{-n/2}+\sum_{i=n+1}^{m-1} k i 2^{-i/2}+2^{-n/2}
\\  \leq & 7^{7!} k n^7 2^{-n/2}
\\ \leq & 9^{9!} k \frac{d(x,y)^{1/2}}{\log(d(x,y))^7},
\end{align*}
using $d(x,y)\geq 2^{-n+1}$ for the last inequality. Therefore, since $(x,y)$ are arbitrary with $d(x,y)\leq 2^{-N-9}$, $\F$ is almost surely H\"older continuous with any exponent smaller than $1/2$.
\end{proof}
\section{A concentration inequality for supremum of empirical process.}

\begin{lemma} \label{4+e} Let $n\in \N$, $\kappa\geq 2$, $(x_i)_{i\leq n}\in \R^n$. Let $(X_i)_{i\leq n}$ be a family of independent centered, random variables.  For every $k\leq n$ let $S_k:= x\in \R\mapsto \sum_{i\leq k}\1_{x\leq x_i}X_i$. Assume that for every $i\leq n$, $v_i:=\E[X_i^\kappa]^{2/\kappa}<\infty$, and let $V:=\sum_{i\leq n} v_i$.
Then for every $t>0$,
\[ P_n:=\proba\left (\sup_{i\leq n} \| S_i\|_{\infty}>t \right )\leq  C_\kappa \left (\sqrt{V}/t\right )^{\kappa}, \]
where $C_\kappa$ is a constant which depends only on $\kappa$.
\end{lemma}
\begin{remark} By taking $n\to \infty$, the previous lemma also holds when $n=\infty$.
\end{remark}
\begin{proof} 
 Fix $\kappa\geq 2$. We work by induction with trivial initialization $n=0$. Fix $C\in \R^+$, which we chose later. Assume that Lemma \ref{4+e} holds for every $m<n$ with the constant $C$.  Let $\e>0$, which we chose later. Let $k\in \N$ be the smallest integer such that $\sum_{i\leq k} v_i\geq \frac{1}{2}\sum_{i\leq n} v_i$. Note that,
\[ P_n\leq  \proba\left (\sup_{i<k} \| S_i\|_\infty>t\right )  +\proba\left (\| S_k\|_\infty>\e t \right ) +\proba\left (\sup_{k<i\leq n} \|S_i-S_k\|_\infty>(1-\e)t \right ). \]
Then by induction, and by definition of $k$,
\begin{equation} P_n\leq C\left (\sqrt{V/2}/t\right )^\kappa+\proba\left (\| S_k\|_\infty>\e t \right )+C \left (\sqrt{V/2}/((1-\e)t)\right )^\kappa. \label{21/02/1h}\end{equation}

It remains to estimate $\| S_k\|_\infty$. To this end, first note that $(|S_k(x)|^\kappa)_{x\in \R}$ is a sub-martingale, so by Doob's inequality, 
\begin{equation} \proba\left (\| S_k\|_\infty>\e t \right )\leq \max_{x\in \R}\frac{\E[|S_k(x)|^\kappa]}{(\e t)^\kappa}=\frac{\E[|\sum_{i=1}^k X_i|^\kappa]}{(\e t)^\kappa}. \label{23/02/1h}\end{equation}

Furthermore by Marcinkiewicz's inequality (see \cite{Massart} Section 15.4), writing $c_\kappa:=2^{\kappa+1}\left (2\kappa \right)^{\kappa/2}$,
\[\E\left [ \left |\sum_{i=1}^k X_i \right |^\kappa \right] \leq c_\kappa \E\left [\left (\sum_{i=1}^{k} X_i^2 \right)^{\kappa/2} \right],\]
and by Minkowski's inequality
\[\E\left [ \left (\sum_{i=1}^k X_i^2 \right )^{\kappa/2} \right] \leq c_\kappa \left (\sum_{i=1}^k \E\left [(X_i^2)^{\kappa/2} \right ]^{2/\kappa} \right )^{\kappa/2}=c_\kappa V^{\kappa/2}.\]

Therefore, by \eqref{21/02/1h} and \eqref{23/02/1h},
\[ P_n\leq \left (C\frac{1+1/(1-\e)^\kappa}{2^\kappa}+c_\kappa/\e^\kappa \right)  \left (\sqrt{V}/t\right )^{\kappa}. \]
Finally some quick computations show that there exists $C,\e$ such that $(C\frac{1+1/(1-\e)^\kappa}{2^\kappa}+c_\kappa/\e^\kappa)\leq C$, and the desired result follows by induction.
\end{proof}
\begin{remark} If $\kappa\geq 4$, one may chose $\e=1/3$,  $C_\kappa=2.3^\kappa c_\kappa\leq 4.9^\kappa \kappa^{\kappa/2}$. 
\end{remark}

\section{Proof of Lemma \ref{Musique qui tue}, and measurability for $\nu_\Lft$, $\nu_\Fnt$ and $\nu_\Rgt$} \label{SadoMasoSec}
Since the proofs of this section consist in checking some list of cases, we strongly advise the reader to draw those different cases to make the reading less tedious. Also, we use the symbol $\ddagger$ to mean that we proved that a case is absurd. Recall Sections \ref{RtreeDef}, \ref{DefLeftFrontRight}. Let $(\T,d,\rho,u)$ denote a plane $\R$-tree. Let $\L:=\T\times[0,1]$. Let $\nu$ denote a $\sigma$-finite measure on $\T$. 

First $\prec$ is reflexive since for every $\x\in \T\times[0,1]$, 
$\x\too\x$.
Also all elements are comparable: For every $\x=(x,a),\y=(y,b)\in \T\times[0,1]$ such that neither $\x\Rgt \y$ nor $\x\Rgt \y$, $U_{x\wedge y,x,y}=U_{x\wedge y,x,y}$. So $x=x\wedge y$ or $y=x\wedge y$, since otherwise $x,y$ would be connected in $\T\backslash\{x\wedge y\}$. In the first case $\x\too \y$. In the second case, $\y\too \x$.
\begin{proof}[Proof that $\prec$ is antisymmetric] Let $\x=(x,a),\y=(y,b)\in \T\times[0,1]$. Assume that $\x\prec \y$ and $\y\prec \x$, let us prove that $\x=\y$. We have either:
\begin{compactitem}
\item[$\bullet$ $\x\Rgt\y$, $\y\Rgt\x$:] So $u_{x\wedge y,x,a}<u_{x\wedge y,y,b}$ and $u_{x\wedge y,y,b}<u_{x\wedge y,x,a}$. $\ddagger$
\item[$\bullet$ $\x\Rgt\y$, $\y\too\x$:] By $\y\Fnt \x$, $y\too x$, $u_{y,x,a}=b$. Then by $x\wedge y=y$ and $\x\Rgt \y$, $u_{y,x,a}<u_{y,y,b}$. $\ddagger$
\item[$\bullet$ $\x\too\y$, $\y\Rgt\x$:] This case is similar by symmetry. $\ddagger$ 
\item[$\bullet$ $\x\too\y$, $\y\too\x$:] $x\in \llbracket 0,y\rrbracket$ and $y\in \llbracket 0, x\rrbracket$ so $x=y$. Moreover, $a=u_{x,y,b}=b$. \qedhere
\end{compactitem}
\end{proof}
Toward proving that $\prec$ is transitive, it is enough to show that:
\begin{lemma} \label{OSKOUR} For every $\x=(x,a),\y=(y,b),\z=(z,c)\in \T\times[0,1]$:
\begin{compactitem}
\item[(a)] If $\x\Rgt \y$, $\y\Rgt \z$, then $\x\Rgt \z$.
\item[(b)] If $\x\Rgt \y$, $\y\too \z$, then $\x\Rgt \z$.
\item[(c)] If $\x\too \y$, $\y\Rgt \z$, then $\x\Rgt \z$ or $\x\too \z$.
\item[(d)] If $\x\too \y$, $\y\too \z$, then $\x\too \z$.
\end{compactitem}
\end{lemma}
\begin{proof} Toward (a) ($\x\Rgt \y$, $\y\Rgt\z$) note that $x\wedge y, y\wedge z\in \llbracket 0, y\rrbracket$, so we have either:
\begin{compactitem}
\item[$\bullet$ $x\wedge y\in \llbracket 0,y\wedge z \llbracket$:] So $x\wedge z=x\wedge y$. Also, $y$ and $z$ are connected in $\T\backslash \{x\wedge y\}$, so $u_{x\wedge y,y,b}=u_{x\wedge y,z,c}$. Then by $\x\Rgt \y$, $u_{x\wedge y,x,a}<u_{x\wedge y,y,b}$. So $u_{x\wedge z,x,a}<u_{x\wedge z,z,c}$ and $x\Rgt z$. 
\item[$\bullet$ $y\wedge z\in \llbracket 0,x\wedge y \llbracket$:] Similarly, $x\wedge z=y\wedge z$, and $u_{y\wedge z,x,a}=u_{y\wedge z,y,b}$. Then by $\y\Rgt \z$, we have $u_{y\wedge z,y,b}<u_{y\wedge z,z,c}$. So $u_{x\wedge z,x,a}<u_{x\wedge z,z,c}$ and $x\Rgt z$. 
\item[$\bullet$ $x\wedge y=y\wedge z$:] $\,$ \\ $\bullet$ Either $x\wedge z=x\wedge y$, and then $u_{x\wedge y,x,a}<u_{x\wedge y,y,b}$ and $u_{y\wedge z,y,b}<u_{y\wedge z,z,c}$. So $u_{x\wedge z,x,a}<u_{x\wedge z,z,c}$. Therefore $x\Rgt z$. 
\\ $\bullet$ Or $x\wedge z\neq x\wedge y$, so $x$ and $z$ are connected in $\T\backslash \{x\wedge y\}$. So $u_{x\wedge y,x,a}=u_{x\wedge y,z,c}$. However, by $\alpha\Rgt \beta$, $u_{x\wedge y,x,a}<u_{x\wedge y,y,b}$ and by $\beta\Rgt \gamma$, $u_{y\wedge z,y,b}<u_{y\wedge z,z,c}$. $\ddagger$
\end{compactitem}

Toward (b), ($\x\Rgt \y$, and $\y\too\z$). We have either:
\begin{compactitem}
\item[$\circ$ $x=y$:] By $\x\Rgt \y$, $a<b$. Also, by $\y\too\z$, $y\in \llbracket 0,z\rrbracket $, and $u_{y,z,c}=b>a$. Thus, $\x\Rgt \z $. 
\item[$\circ$ $x\neq y$, $x\wedge y=y$:] By $\x\Rgt \y$, $u_{y,x,a}<u_{y,y,b}=b$. Also, by $\y\too\z$, $u_{y,z,c}=b$. Thus, $x$ and $z$ are in different components of $\T\backslash \{y\}$,  so $x\wedge z=y$. Hence, $\x\Rgt \z$. 
\item[$\circ$ $x\neq y$, $x\wedge y\neq y$:] $y\too \z$ so $y$ and $z$ are connected in $\T\backslash \{x\wedge y\}$, thus $x\wedge y=x\wedge z$ and $u_{x\wedge y, y,b}=u_{x\wedge y, z,c}$. Then by $\x\Rgt \y$, $u_{x\wedge y,x,a}<u_{x\wedge y,y,b}$. Hence, $\x\Rgt \z$. 
\end{compactitem}

Toward (c), we have $\x\too\y$, and $\y\Rgt\z$ so $x, y\wedge z \in \llbracket 0, y\rrbracket$. Then we have either:
\begin{compactitem}
\item[$\bullet$ $y\wedge z\in \rrbracket x,y\rrbracket$:] So $y,z$ are connected in $\T\backslash\{x\}$. Thus $u_{x,z,c}=u_{x,y,b}$. Then by $\x\too \y$, $u_{x,y,b}=a$. Hence $\x\too \z$. 
\item[$\bullet$ $y\wedge z=x$:] $x\wedge z=x$. By $\x\too \y$, $u_{x,y,b}=a$, and by $\y\Rgt \z$, $u_{x,y,b}<u_{x,z,c} $. So $\x\Rgt \z$. 
\item[$\bullet$ $y\wedge z\in \llbracket 0,x\llbracket$:] Then by $x\in \llbracket 0,y\rrbracket$, $x\wedge z=y\wedge z$, and $x$ and $y$ are connected in $\T\backslash \{y\wedge z\}$. Thus $u_{y\wedge z, x,a}=u_{y\wedge z, y,b}$. Then by $\y\Rgt \z$, $u_{y\wedge z, y,b}<u_{y\wedge z, z,c}$. Hence, $x\Rgt z$.
\end{compactitem}

Toward (d), ($\x\too y$, $\y\too\z$), we have either:
\begin{compactitem}
\item[$\circ$ $x=y$:] By $\x\too\y$, $b=u_{x,y,b}=a$ so $\x=\y$ and $\x\too \z$. 
\item[$\circ$ $x\neq y$]: By $x\too y$ and $y\too z$, $y$ and $z$ are connected in $\T\backslash \{x\}$. So $u_{x,z,c}=u_{x,y,b}=a$, $\x\too \z$. \qedhere
\end{compactitem}
\end{proof}
Also by a simple symmetry argument, (change the angles by $u\mapsto 1-u$) by Lemma \ref{OSKOUR} (b),
\begin{lemma} \label{OSGOUR} With the same notations, if $\y\too \z$, $\y\Rgt \x$ then $\z\Rgt \x$.
\end{lemma}
\begin{lemma} \label{Measurable} $\nu_{\Lft}$, $\nu_{\Fnt}$, $\nu_{\Rgt}$ are well defined. 
\end{lemma}
\begin{proof} First let $\mathcal M:=\{x\in \T,\deg(x)>2\}\cup \{\rho\}$. $\mathcal M$ is countable since $\T$ is separable. Then for every $x\in \T$, and $a\in [0,1]$, let 
\[ \connex_{x,a}:=\{y\in \T\backslash \{x\}, u_{x,y}=a\}. \]
This set is open so measurable. Also, for every $x\in \T$, and $a\in ]0,1[$ let $\connex_{x,\leq a}:=\bigcup_{b\leq a, b\neq 0} \connex_{x,b}$. Define similarly, $\connex_{x,<a},\connex_{x,\geq a}, \connex_{x,>a}$. Since $\T$ is separable all those sets are a countable union of open set and so are measurable. 

Then by simply rewriting the definition of $\Lft$, for every $\x=(x,u)\in \L$,
\begin{align*}\{\y\in \L, \y\Rgt \x\}
 = &\left (\bigcup_{z\in \mathcal M\cap \llbracket 0,x\llbracket} \connex_{z,< u_{z,x}}\times[0,1]\right )\cup \left (\bigcup_{z\in \mathcal M\cap \llbracket 0,x\rrbracket} \{z\} \times[0,u_{z,x})\right )
\\ & \cup \left ((\llbracket 0,x\llbracket\backslash \mathcal M)\times[0,1/2)\right )\cup ( \{x\}\times [0,u))\cup \connex_{x,<u}.\end{align*}
So $\{\y\in \L, \y\Rgt \x\}$ is measurable as a countable union of measurable set. Similarly, 
\begin{align*}\{\y\in \L, \x\Rgt \y\}
 = &\left (\bigcup_{z\in \mathcal M\cap \llbracket 0,x\llbracket} \connex_{z,> u_{z,x}}\times[0,1]\right )\cup \left (\bigcup_{z\in \mathcal M\cap \llbracket 0,x\rrbracket} \{z\} \times(u_{z,x},1]\right )
\\ & \cup \left ((\llbracket 0,x\llbracket\backslash \mathcal M)\times(1/2,1]\right )\cup ( \{x\}\times (u,1])\cup \connex_{x,>u},\end{align*}
is also measurable. Finally, $\{\y\in \L, \x\too \y\}=\connex_{x,u}\cup\{(x,u)\}$ is measurable.
\end{proof}
Also for every $\x\in \L$, 
\[ \{\z\in \L:\z \too \x \}=\left (\bigcup_{z\in \mathcal M\cap \llbracket 0,x\rrbracket} \{z,u_{z,x}\}\right )\cup  \left ((\llbracket 0,x\llbracket\backslash \mathcal M)\times\{1/2\}\right ) \]
 is measurable.
\begin{lemma} \label{Measurable2} $\{\x,\y\in \L, \x\too \y\}$, $\{\x,\y\in \L, \x\Rgt \y\}$ are measurable sets for the product topology.
\end{lemma}
\begin{proof}  First note that $\bigcup_{x,y\in \T}\rrbracket x,y \llbracket$ is dense on $\T$, so since $\M$ is countable, $\{x\in \T,\deg(x)=2\}$ is dense. Then since $\T$ is separable, there is a countable dense set $\mathcal N$ of $\T$, and we may assume that $\mathcal N\subset \{x\in \T,\deg(x)=2\}$. Moreover, note that for every $x\neq y \in \T$, $\rrbracket x,y\llbracket \cap (\mathcal M \cup \mathcal N)\neq \emptyset$.

Then, let for $z\in \T$, $C'_{z,0}:=\{(x,a)\in \T\times[0,1]:(x,a)\too (z,0), x\neq z\}$. The set
\[\{\x,\y\in \L, \x\too \y\}=\{ \{x,a\}\times\{x,a\},x\in \T,a\in [0,1[\}\cup \left (\bigcup_{z\in \mathcal M\cup \mathcal N} C'_{z,0}\times \connex_{z,\leq 1} \right ), \]
is  measurable as a countable union of measurable set.

Also, by considering several cases depending on the position of $x\wedge y$, we get: (The two last cases are here to deal with the cases where $x\in \llbracket 0,y\llbracket$ or $y\in \llbracket 0,x\llbracket$.)
\begin{align*} & \{\x,\y\in \L, \x\Rgt \y\} =  \{\{x,u\}\times\{x,v\}, x\in \T, u<v\}\cup \left (\bigcup_{z\in \mathcal \M,u\in \mathbb Q} \connex_{z,<u}\times \connex_{z,>u}\right )
\\ & \cup \left  (\bigcup_{z\in \mathcal N} \{\x,\x\Rgt (z,1/2)\}\times \{\y,(z,1/2)\too \y\} \right ) \cup \left (\bigcup_{z\in \mathcal N} \{\x,(z,1/2)\too\x \} \times \{\y,(z,1/2)\Rgt \y\} \right ). \end{align*}
Thus $\{\x,\y\in \L, \x\Rgt \y\}$ is measurable as a countable union of measurable set.
\end{proof}
\end{document}